\documentclass[11pt,a4paper,reqno]{amsart}%
\usepackage{amsthm,amsmath,amsfonts,amssymb,xcolor,amsxtra,appendix,bookmark,dsfont,bm,mathrsfs}
\usepackage[margin=25mm]{geometry}

\usepackage[latin1]{inputenc}
\usepackage{color} 
\usepackage{amssymb}\usepackage{graphicx}
\usepackage{tikz}
\usepackage{hyperref}
\theoremstyle{plain}
\newtheorem{theorem}{Theorem}
\newtheorem{lemma}[theorem]{Lemma}
\newtheorem{corollary}[theorem]{Corollary}
\newtheorem{proposition}[theorem]{Proposition}

\newtheorem{definition}{Definition}

\theoremstyle{remark}
\newtheorem{remark}[theorem]{Remark}

\allowdisplaybreaks

\begin{document}
\title[Wave turbulence theory  for elastic beam waves]{On the wave turbulence theory: ergodicity  for the elastic beam wave equation}

\author[B. Rumpf]{Benno Rumpf }
\address{Department of Mathematics, Southern Methodist University, Dallas, TX 75275, USA} 
\email{brumpf@mail.smu.edu}

\author[A. Soffer]{Avy Soffer}
\address{Mathematics Department, Rutgers University, New Brunswick, NJ 08903 USA.} 
\email{soffer@math.rutgers.edu}

\author[M.-B. Tran]{Minh-Binh Tran}
\address{Department of Mathematics, Texas A\&M University College Station, TX 77843, USA}
\email{minhbinh@tamu.edu} 
\maketitle
\begin{abstract}  We analyse a 3-wave kinetic equation, derived from the elastic beam wave equation on the lattice. The ergodicity  condition states that two distinct wavevectors are supposed to be connected by a finite number of collisions.  In this work, we  prove that  the ergodicity condition is violated and  the equation domain is broken into disconnected domains, called no-collision and collisional invariant regions.  If one starts with a general initial condition, whose energy is finite, then in the long-time limit, the solutions of the 3-wave kinetic equation remain unchanged on the no-collision region and relax to local equilibria on the disjoint collisional invariant regions. To our best knowledge, this is the first time that the violation of the ergodicity condition is observed and proved  for a kinetic equation.   
 \end{abstract}

{\bf Keyword:}   
wave turbulence, convergence to equilibrium, ergodicity condition


\tableofcontents
\section{Introduction}
 Having the origin in the works of  Peierls \cite{Peierls:1993:BRK,Peierls:1960:QTS},  Hasselmann \cite{hasselmann1962non,hasselmann1974spectral},  Benney-Saffman-Newell \cite{benney1969random,benney1966nonlinear},  Zakharov \cite{zakharov2012kolmogorov},   wave kinetic  equations have been shown to play  important roles in a vast range of physical examples and this is why a huge and still growing number of situations have used WT theory: inertial waves due to rotation; Alfv\'en wave turbulence in the solar wind; waves in plasmas of fusion devices; and many others, as discussed in the books of   Zakharov et.al. \cite{zakharov2012kolmogorov},  Nazarenko
\cite{Nazarenko:2011:WT} and the review papers of Newell and Rumpf \cite{newell2011wave,newell2013wave}.

We consider the quadratic elastic  beam wave  equation (Bretherton-type equation) (see  Bretherton \cite{bretherton1964resonant}, Benney-Newell \cite{benney1967propagation}, Love \cite{love2013treatise})
\begin{equation}
	\label{QuadraticNLS}\begin{aligned}
&	\frac{\partial^2\psi}{\partial T^2}(x,T) \ + \ (\Delta+c)^2 \psi(x,T)   \ + \lambda\psi^2(x,T) \ = \ 0,\\
\psi(x,0) \ & = \ \psi_0(x), \ \ \frac{\partial\psi}{\partial T}(x,0) \ = \ \psi_1(x),\end{aligned}
\end{equation}
for $x$ being on 
{ $\mathbb{Z}^3$}, $T\in\mathbb{R}_+$, $c\in\mathbb{R}$ is some real constant, $\lambda$ is a small constant describing the smallness of the nonlinearity.  Equations of type  \eqref{QuadraticNLS} have been widely studied in  control theory, and have been shown to have a Schr\"odinger structure (see, for instance, Burq \cite{burq1993controle}, Fu-Zhang-Zuazua \cite{fu2006optimality}, Haraux \cite{haraux1989series}, Lebeau \cite{lebeau1992control}, Lions \cite{lions1988controlabilite},  and Zuazua-Lions \cite{zuazua1987controlabilite}.) The analysis of \eqref{QuadraticNLS}  is also an interesting mathematical question of current interest (see, for instance,  Hebey-Pausader \cite{hebey2008introduction}, Levandosky-Strauss \cite{levandosky2000time}, Pausader \cite{pausader2007scattering} Pausader-Strauss \cite{pausader2009analyticity}.)

We obtain the 3-wave kinetic equation
\begin{equation}\label{PhononEqC}
\begin{aligned}
\partial_t f(k,t) \ & =  \ Q_c[f](k),\ \ \ \ f(k,0) \ = \ f_0(k),\ \  \forall k\in\mathbb{T}^3,\\
Q_c[f](k)\ & = \ \int_{\mathbb{T}^6}K(\omega,\omega_1,\omega_2)\delta(k-k_1-k_2)\delta(\omega-\omega_1-\omega_2)[f_1f_2-ff_1-ff_2]\mathrm{d}k_1\mathrm{d}k_2\\
& - \ 2\int_{\mathbb{T}^6}K(\omega,\omega_1,\omega_2)\delta(k_1-k-k_2)\delta(\omega_1-\omega-\omega_2)[f_2f-ff_1-f_1f_2]\mathrm{d}k_1\mathrm{d}k_2,
\end{aligned}
\end{equation} 
where $K(\omega,\omega_1,\omega_2)=[\sqrt{8}\omega(k)\omega(k_1)\omega(k_2)]^{-1}$, with
\begin{equation*}
	\label{DispersionRelation}
	\omega(k) \ =  \omega_0 \ + \ \sum_{j=1}^32\Big(1-\cos(2\pi k^j)\Big),
\end{equation*} 
and $\mathbb{T}^d$ is the periodic torus $[0,1]^d$.

One of the main challenges in understanding the behaviors of solutions to the 3-wave kinetic equations is the so-called {\it ergodicity}, which is quite typical for 3-wave processes.  Ergodicity has played a very important role and has a long history in physics \cite{aoki2002nonequilibrium,lee2005thermal,lefevere2004perturbative,lepri2000memory} and we refer to the lecture notes   \cite{Spohn:TPB:2006}[Section 17] for a more detailed discussion. To define ergodicity, we will need the concept of  the connectivity between two wave vectors $k$ and $k'$, which we briefly discuss   here, leaving the precise definition for later. Given a wave vector $k$, a wave vector $k'$ is understood to be connected to $k$ in a collision if either $\omega(k')=\omega(k)+\omega(k'-k)$,  $\omega(k)=\omega(k')+\omega(k-k')$, or $\omega(k+k')=\omega(k)+\omega(k')$.

{\bf Ergodicity Condition (E):} 
{\it For every $k,k'\in\mathbb{T}^3\backslash\{0\},$ there is a finite sequence of collisions such that $k$ is connected to $k'$. }

When the  Ergodicity Condition (E) is violated, the system is partitioned into smaller subsystems   which are dynamically disconnected and each subsystem 
thermalizes by itself.

It was shown that (see \cite{Spohn:TPB:2006}) under the Ergodicity Condition (E), the only stationary solutions of the spatially homogeneous Boltzmann equations \eqref{PhononEqC} take the forms
$$\frac{1}{\beta \omega(k)},$$
in which $\beta$ can be computed via the conservation laws. 

The aim of this work is to develop a rigorous analysis  and prove that the  Ergodicity Condition (E) is violated for the equation \eqref{PhononEqC}. We will show that  the domain of integration is broken into disconnected domains. Those subregions are then proved to be dynamically disconnected. There is one region, in which if one starts with any initial condition, the solutions remain unchanged as time evolves.  In general, the equilibration temperature will differ from region to region. We call it the ``no-collision region''. The rest of the domain is divided into disconnected regions, each has their own local equilibria. If one starts with any initial condition, whose energy is finite on one subdomain, the solutions will relax to the local equilibria of this subregion, as time evolves, and as thus each subsystem 
thermalizes by itself. Those subregions are named ``collisional invariant regions'', due to the fact that we can rigorously establish  unique local collisional invariants on each of them, using the conservation of energy. This confirms  Spohn's prediction and enlightening physical intuitions \cite{Spohn:TPB:2006} on the behavior of 3-wave systems. To our best knowledge, this is the first example in which the important ergodicity condition is violated for a kinetic equation. { We also remark that  the 3-wave kinetic equation considered in this work describes the translation
invariant system and the results
proven  (decomposition of the frequency space $\mathbb{T}^d$ into disjoint
equivalence classes under connectedness via collisions, that are
invariant under the flow) do not hold for the spatially inhomogenous version of the equation.
}

In addition to 3-wave kinetic equations, 4-wave kinetic equations have also played an important role in wave turbulence and have been first studied in the work of Escobedo and Velazquez in \cite{EscobedoVelazquez:2015:FTB,EscobedoVelazquez:2015:OTT} as well as  several other works \cite{ampatzoglou2024inhomogeneous,collot2024stability,dolce2024convergence,escobedo2024instability,germain2017optimal,menegaki20222,staffilani2024condensation,staffilani2024energy}. 

{\bf Acknowledgements.}  B. R.  is  funded in part by  a grant from the Simons
Foundation (No. 430192). A. S. is in part by the  Simons Foundation grant number  851844 and NSF-DMS grant number 220931. 
M.-B. T is  funded in part by  a   Humboldt Fellowship,   NSF CAREER  DMS-2303146, and NSF Grants DMS-2204795, DMS-2305523,  DMS-2306379. We would like to thank Prof. Herbert Spohn, Prof. Gigliola Staffilani and Prof. Enrique Zuazua for fruitful discussions on the topic. We would like to thank the referees for the important remarks.

\section{From the Bretheton equation to the 3-wave kinetic  equation}\label{Sec:Derivation}
 { For the sake of completeness, in this section, we recall the formal derivation of the 3-wave kinetic  equation from the Bretheton equation for the general dimension $d>2$. 
We follow  the same strategy   of \cite{Spohn:TPB:2006,LukkarinenSpohn:WNS:2011} to put the equation on a lattice
\begin{equation}
	\label{Lattice}
	\Lambda \ = \ \Lambda(D) \ = \ \left\{1, \dots,2D
	\right\}^d, 
\end{equation}}
for some constant $D\in\mathbb{N}$.


The  discretized equation   is now
\begin{equation}
	\label{LatticeDynamics}\begin{aligned}
		\partial_{TT}\psi(x,T) \ = &  \ \ -\sum_{y\in  \Lambda} O_1(x-y)\psi(y,T) \ - \ \lambda ({\psi(x,T)})^2, \\
		\psi(x,0) \ & = \ \psi_0(x), \ \partial_{T}\psi(x,0) \  = \ \psi_1(x), \ \forall (x,T)\in\Lambda\times \mathbb{R}_+,\end{aligned}
\end{equation}
where $O_1(x-y)$ is a finite difference operator that we will express below in the Fourier space.	We remark that a similar beam dynamics of  non-acoustic chains has also been considered in \cite{basile2016thermal}[Section 7].
 To obtain the lattice dynamics,  we introduce  the Fourier transform 
\begin{equation}
	\label{Def:Fourier}\hat  \psi(k)=\sum_{x\in\Lambda} \psi(x) e^{-2\pi {\bf i} k\cdot x}, \quad k\in \Lambda^* = \Lambda^*(D) =  \left\{0,\cdots,\frac{2D}{2D+1}\right\}^d,
\end{equation}
{ which is a subset of the $d$-dimensional torus $[0,1]^d$. }
 We also define the mesh size to be \begin{equation}\label{Mesh} h^d=\left(\frac{1}{2D+1}\right)^d.\end{equation} 
At the end of this standard procedure, \eqref{LatticeDynamics}  can  be rewritten in the Fourier space as a system of ODEs
\begin{equation}
	\label{LatticeDynamicsFourier}\begin{aligned}
		\partial_{TT}\hat\psi(k,T) \ = \  &   -\omega(k)^2\hat\psi(k,T)\\
		& \ - \lambda \sum_{k_1,k_2\in\Lambda^*}{\hat\psi(k_1,T)}\delta(k-k_1-k_2)\hat\psi(k_2,T),\\
		\hat\psi(k,0) \ & = \ \hat\psi_0(k) ,\ \ \ \partial_{T}\hat\psi(k,0) \ = \ \hat\psi_1(k) ,  \end{aligned}
\end{equation}
where the  dispersion relation  takes the  discretized form 
\begin{equation}
	\label{NearestNeighbordA}
	\omega_k \ = \ \omega(k) \ = \ \sin^2(2\pi h k^1) + \cdots + \sin^2(2\pi  k^d) + c,
\end{equation}
with $k=(k^1,\cdots,k^d)$.

We define the inverse Fourier transform to be 
\begin{equation}\label{Def:FourierInverse}
	f(x)= \sum_{k\in\Lambda_*} \hat  f(k) e^{2\pi {\bf i} k\cdot x}.
\end{equation}
We also use the following notations
\begin{equation}
	\label{Shorthand1}
	\int_{\Lambda}\mathrm{d}x \ = \  h^d\sum_{x\in\Lambda},\ \ \ \ 
	\langle f, g\rangle \ = \ h^d \sum_{x\in\Lambda}f(x)^* g(x),
\end{equation}
where   if $z\in \mathbb{C}$, then $\bar{z}$ is the complex conjugate, as well as the  Japanese bracket
\begin{equation}
	\label{JapaneseBracket}
	\langle x\rangle \ = \ \sqrt{1+|x|^2}, \ \ \forall x\in\mathbb{R}^d.
\end{equation}
And
\begin{equation}
	\sum_{k\in\Lambda^*} \ = \ \int_{\Lambda^*}\mathrm{d}k.
\end{equation}

{Moreover, for any $N\in\mathbb{N}\backslash\{0\}$, following precisely \cite{LukkarinenSpohn:WNS:2011}[equation (2.9)], we define the delta function $\delta_N$ on $(\mathbb{Z}/N)^d$ as
\begin{equation}
	\label{Def:Delta}\delta_N(k) = |N|^d\mathbf{1}(k \mbox{ mod } 1 \ = \ 0), \ \ \ \forall k\in (\mathbb{Z}/N)^d.
\end{equation}
In our computations, we omit the sub-index $N$ and simply write
\begin{equation}
	\label{Def:Delta1}\delta(k) = |N|^d\mathbf{1}(k \mbox{ mod } 1 \ = \ 0), \ \ \ \forall k\in (\mathbb{Z}/N)^d.
\end{equation} 
\begin{remark}
	Note that, the above definition of the discrete delta function follows the classical definition of Lukkarinen-Spohn \cite{LukkarinenSpohn:WNS:2011}[equation (2.9)], commonly used in the derivation of wave kinetic equations. The factor $|N|^d$ is needed as it guarantees the convergence of the discrete delta function to the continuum delta function in the limit of $N$ going to $\infty$. 
\end{remark}}
Equation \eqref{LatticeDynamicsFourier} can now be expressed as a coupling system
\begin{equation}
\label{LatticeDynamicsSystem}\begin{aligned}
\frac{\partial }{\partial T}{q}(k,T) \ & = \ {p}(k,T),\\
\frac{\partial }{\partial T}{p}(k,T) \ & = \ -\omega^2(k){q}(k,T)\\
&\ \ \  \ -  \lambda\int_{(\Lambda^*)^2}\mathrm{d}k_1\mathrm{d}k_2\delta(k-k_1-k_2){q}(k_1,T){q}(k_2,T), \\
q(k,0) \ & = \ \hat\psi_0(k), \ \ p(k,0) \  = \ \hat\psi_1(k), \ \ \ \ \forall (k,T)\in\Lambda^*\times \mathbb{R}_+,\end{aligned}
\end{equation}
which, under Spohn's transformation  (see \cite{Spohn:TPB:2006})
\begin{equation}\label{Tranformation1}
a(k,T) \ =  \ \frac{1}{\sqrt{2}}\Big[{\omega(k)^\frac12}{q}(k,T)\ +  \ \frac{i}{{\omega(k)^\frac12}}{p}(k,T)\Big],
\end{equation}
leads to the following system of ordinary differential equations
\begin{equation}
\begin{aligned}\label{ODEs}
\frac{\partial}{\partial T} a(k,T) \ & =  \ i\omega(k)a(k,T) \ - \ i\lambda\int_{(\Lambda^*)^2}\mathrm{d}k_1\mathrm{d}k_2\delta(k-k_1-k_2)\times\\
&  \times [8\omega(k)^21\omega(k_1)^2\omega(k_2)^2]^{-\frac12}\Big[a(k_1,T) \ + \ a^*(-k_1,T)\Big]\Big[a(k_2,T) \ + \ a^*(-k_2,T)\Big],\\
a(k,0) \ & = \ a_0(k) \ = \ {\frac{1}{\sqrt{2}}\Big[{\omega(k)}{q}(k,0)\ +  \ \frac{i}{{\omega(k)}}{p}(k,0)\Big]},  \forall (k,T)\in\Lambda^*\times \mathbb{R}_+.
\end{aligned}
\end{equation}
%
%


In order to absorb the quantity ${ i}\omega(k)\hat{a}(k,\sigma,T)$ on the right hand side of the above system, we set
\begin{equation}
	\label{HatA}
	\alpha(k,T) \ = \ {a}(k,T)e^{-{ i} \omega(k)T}.
\end{equation}
 The following system can be now derived for  $\alpha_T(k)$
\begin{equation}
	\begin{aligned}\label{StartPoint}
	 &	\frac{\partial}{\partial T} \alpha(k,T) \ =  - \ { i}\sigma\lambda \sum_{k_1,k_2\in\Lambda^*}\delta(k-k_1-k_2)[8\omega(k)^2\omega(k_1)^2\omega(k_2)^2]^{-\frac12}\times\\
		& \ \  \ \ \  \ \times\Big[ \alpha(k_1,T) \ + \  \alpha^*(-k_1,T)\Big]\Big[ \alpha(k_2,T) \ + \  \alpha^*(-k_2,T)\Big]e^{-{ i}T(-\omega(k_1)-\omega(k_2)+\omega(k))}.
	\end{aligned}
\end{equation}

Consider the two-point  correlation function 
\begin{equation}
	\label{Expectation}
	f_{\lambda,D}(k,T) \ = \ \langle {\alpha}_{T}(k,-1){\alpha}_{T}(k,1)\rangle.
\end{equation}
In the limit of  $D\to \infty$, $\lambda\to 0$ and $T=\lambda^{-2}t=\mathcal{O}(\lambda^{-2})$, the two-point correlation function 
$f_{\lambda,D}(k,T)$ has the limit
$$\lim_{\lambda\to 0, D\to\infty} f_{\lambda,D}(k,\lambda^{-2}t) = f(k,t)$$
which solves the 3-wave  equation \eqref{PhononEqC}, by the standard formal derivation of \cite{Spohn:TPB:2006}.  
\begin{remark}\label{Remarkdelta}
	As a consequence of the definition \eqref{Def:Delta}-\eqref{Def:Delta1}, the delta function $\delta(k-k_1-k_2)$ in the collision operator of \eqref{PhononEqC} means that there exists a vector $z\in\mathbb{Z}^d$ such that $k=k_1+k_2+z$. 
\end{remark}

\section{Main results}
Let us first normalize the dispersion $\omega$ as
\begin{equation}
	\label{DispersionRelation}
	\omega(k) \ =  \omega_0 \ + \ \sum_{j=1}^32\Big(1-\cos(2\pi k^j)\Big),
\end{equation}
where $2<\omega_0<3$, and $k=(k^1,k^2,k^3)$. This will result in  an addition factor $4$ comparison to the dispersion relation defined in \eqref{NearestNeighbordA}, leading to a factor of $4$ to the kernel    $K(\omega,\omega_1,\omega_2)$.
In our proof,  we suppose  $K(\omega,\omega_1,\omega_2)$ is $[\omega(k)\omega_1(k)\omega_1(k)]^{-1}$ for the sake of simplicity.

For $\infty>  m\ge 1$, let $\mathcal{S}$ be a Lebesgue measurable subset of $\mathbb{T}^3$ such that its measure is strictly positive, we introduce the function space ${L}^m(\mathcal{S})$, defined by the norm
\begin{equation}\label{Lm}
	\|f\|_{{L}^m(\mathcal{S})}:=\left(\int_{\mathcal{S}} |f(p)|^m \mathrm{d}p\right)^\frac{1}{m}.
\end{equation}
In addition, we also need the space ${L}^\infty(\mathcal{S})$, defined by the norm
\begin{equation}\label{Linf}
	\|f\|_{{L}^\infty(\mathcal{S})}:=\mathrm{esssup}_{p\in \mathcal{S}} |f(p)|.
\end{equation}
We denote by $C^m(\mathcal{S})$, $m=0,1,2,\dots$, the restrictions of all continuous and $m$-time differentiable functions on $\mathbb{T}^3$ onto $\mathcal{S}$. The space $C^0(\mathcal{S})=C(\mathcal{S})$ is endowed with the usual sup-norm \eqref{Linf}.
In addition, for any normed  space $(Y,\|\cdot\|_Y)$, we define
\begin{equation}\label{CY} 
	C([0,T),Y):=\Big\{F:[0,T) \to Y \; \big| F \mbox{ is continuous from $[0,T)$ to $Y$}\Big\}
\end{equation}
and
\begin{equation}\label{C1Y} 
	C^1((0,T),Y):=\Big\{F:(0,T) \to Y \; \big| F \mbox{ is continuous and differentiable from $(0,T)$ to $Y$}\Big\},
\end{equation}
for any $T\in(0,\infty]$. The above definitions can also be extended to the spaces $C([0,T],Y)$, $C^1((0,T],Y)$ for any $T\in(0,\infty)$.

Let us state our main theorem.
\begin{theorem}
	\label{Thm:MainClassical}
	Under the assumption that there exists a positive, classical solution $f$ in $C([0,\infty), C^1(\mathbb{T}^3))\cap C^1((0,\infty), $ $C^1(\mathbb{T}^3))$ of \eqref{PhononEqC}, with the initial condition $f_0\in C(\mathbb{T}^3)$, $f_0(k)\ge 0$ for all $k\in\mathbb{T}^3$. 
	
{There exist  subsets $ \mathfrak{V},\mathfrak{I}\subset\mathbb{T}^3$ such that the torus $\mathbb{T}^3$ can be decomposed into disjoint subsets as follows
{	\begin{equation}
		\label{Thm:MainClassical:1}
		\mathbb{T}^3 \ = \ \mathfrak{I} \cup\bigcup_{x\in \mathfrak{V}}\mathcal{S}(x),
	\end{equation}}
where $\mathcal{S}(x)\cap\mathcal{S}(y)=\emptyset$ and  $\mathcal{S}(x)\cap \mathfrak{I}= \emptyset$ for $x,y\in \mathfrak{V}$.}  The set $\mathfrak{I}$ is not empty and is called the ``no-collision region''. The set $\mathcal{S}(x)$ is called the ``collisional-invariant region''.   The solution $f$ behaves differently on each sub-region.
	\begin{itemize}
		\item[(I)] On $\mathfrak{I}$ the solution stays the same for all time 
		$$f(t,k)=f_0(k),\ \ \ \  \forall t\ge 0,\ \ \ \ \forall k\in  \mathfrak{I}.$$
		\item[(II)]
		{ Let  $x$ be in $\mathfrak{V}$, suppose that the Lebesgue measure $\mathscr{L}(\mathcal S(x)) $ of $\mathcal S(x)$  is strictly positive, let $ E_x \in \mathbb{R}_+$ be a   constant and assume further that it is indeed the local    energy of the initial condition on $\mathcal{S}(x)$
		$$ \int_{\mathcal{S}(x)}f_0(k) \omega(k)\mathrm{d}k = E_x.$$ Suppose that  
		\begin{equation}\label{Thm:MainClassical:1}
			\begin{aligned}
			\frac{1}{a_x}	\int_{\mathcal{S}(x)}\mathrm{d}k\ = \ \frac{\mathscr{L}(\mathcal S(x))}{a_x} \ = & \ E_x,
			\end{aligned}
		\end{equation}}
		with $a_x\in\mathbb{R}_+$; the local equilibrium on the collision invariant region $\mathcal{S}(x)$ can be uniquely determined as
		\begin{equation}\label{Thm:MainClassical:3}
			\frac{1}{a_x\omega(k)}.
		\end{equation}
		Then, the following limits always holds true
		\begin{equation}\label{Thm:MainClassical:4}
			\lim_{t\to\infty}\left\|f(t,k)- \frac{1}{a_x\omega(k)}\right\|_{L^1(\mathcal{S}(x))}=0.
		\end{equation}
		and
		\begin{equation}\label{Thm:MainClassical:5}
			\lim_{t\to\infty}\left|\int_{\mathcal{S}(x)}\ln[f]\mathrm{d}k -\int_{\mathcal{S}(x)}\ln\left[\frac{1}{a_x\omega(k)}\right]\mathrm{d}k\right|=0.
		\end{equation}
		If, in addition, there is a positive constant $M^*>0$ such that $f(t,k)<M^*$ for all $t\in[0,\infty)$ and for all $k\in\mathcal{S}(x)$, then 
		\begin{equation}\label{Thm:MainClassical:6}
			\lim_{t\to\infty}\left\|f(t,\cdot)-\frac{1}{a_x\omega(k)}\right\|_{L^p(\mathcal{S}(x))}=0,
			\ \ \ \ \forall p\in[1,\infty).
		\end{equation}
		If we assume further that $f_0(k)>0$ for all $k\in\mathcal{S}(x)$, there exists a constant $M_*$ such that $f(t,k)>M_*$ for all $t\in[0,\infty)$ and for all $k\in\mathcal{S}(x)$.
	\end{itemize}
\end{theorem}

\begin{remark}
	In the above theorem, we assume the well-posedness of the equation. As this piece of analysis is quite subtle and long, we reserve it for a separate paper.
\end{remark}
\begin{remark}
	Notice that, according to our result, the torus $\mathbb{T}^3$ can be decomposed into disjoint subsets as follows
	\begin{equation}
		\label{Thm:MainClassical:1}
		\mathbb{T}^3 \ = \ \mathfrak{I}\cup \bigcup_{x\in \mathfrak{V}}\mathcal{S}(x),
	\end{equation}
	where  $\mathcal{S}(x)\cap\mathcal{S}(y)=\emptyset$ and  $\mathcal{S}(x)\cap \mathfrak{I}= \emptyset$ for $x,y\in \mathfrak{V}$.  However, those  disjoint subsets might
	 be topologically disconnected sets. 
\end{remark}
{
\begin{remark}
Since our solutions are assumed to be are regular and non-measured,	in (II) of the above theorem, the condition that $\mathscr{L}(\mathcal S(x)) >0$ is essential. When   $\mathscr{L}(\mathcal S(x)) =0$, it follows that 	$\frac{1}{a_x}\int_{\mathcal{S}(x)}\mathrm{d}k\ = \ \int_{\mathcal{S}(x)}f_0(k) \omega(k)\mathrm{d}k  \ =   \ 0,$ and those cases are negligible due to the assumption on our solutions. The case when $\mathscr{L}(\mathcal S(x)) =0$ is more interesting when the solutions are  measures and this problem is being investigated and will be reported in an upcoming work. 
\end{remark}}

The above two theorems assert that those subregions are all non-empty. In the no-collision region $\mathfrak{I}$,  any wavevector $k\in \mathfrak{I}$ is totally disconnected to other wavevectors, and thus the solutions on $\mathfrak{I}$ do not change as time evolves. In each of the collisional invariant regions $\mathcal{S}(x)$, as time goes to infinity, the solutions  converge in the $L^1(\mathcal{S}(x))$-norm  to $ \frac{1}{a_x\omega(k)  }$. In the classical case, to obtain the convergence, we need more regularity on the solutions: we assume that the solutions are in  $C([0,\infty), C^1(\mathbb{T}^3))\cap C^1((0,\infty), $ $C^1(\mathbb{T}^3))$.

Let us also mention that this asymptotic behavior of the solutions to this 3-wave equations is very different from what is observed in spatially homogeneous and isotropic capillary or acoustic kinetic wave equations. It is showed in \cite{soffer2019energy} that if one looks for a solution whose energy is a constant for all time to one of these isotropic capillary/acoustic kinetic wave equations, then this solution can exist only up to a finite time, after this time, some energy is lost to infinity. In other words, the solution exhibits the so-called energy cascade phenomenon.

\section{The  analysis of the 3-wave kinetic  equation}\label{Sec:Classical}
In our proof, as discussed above, we suppose  $K(\omega,\omega_1,\omega_2)$ is $[\omega(k)\omega_1(k)\omega_1(k)]^{-1}$ for the sake of simplicity.  We assume, when needed that the Lebesgue measure of each collisional region  $\mathscr{L}(\mathcal S(x)) $ is strictly positive due to Remark 6.
\subsection{No-collision, collisional regions and the 3-wave kinetic operator on these local disjoint sets}

{In this section $x, y, z$ are now used for frequency vectors, in contrast
to the previous discussion in equation (\eqref{QuadraticNLS}) and the previous sections. Therefore,  that the notations will now
be unlinked from what they were in prior sections.}
\subsubsection{Collisional invariant regions}\label{CollisionInvariantRegions}
For a vector $x=(x^1,x^2,x^3)\in\mathbb{T}^3$, we say that the wave vector $x$ is connected to the wave vector $y=(y^1,y^2,y^3)\in\mathbb{T}^3$ by a {\bf forward collision} if and only if
\begin{equation}\label{ForwardCollision}
\mathfrak{F}^f_x(y) \ := \ \sum_{j=1}^32[\cos(2\pi (y_j-x_j)) \  + \ \cos(2\pi x_j) \ - \ \cos(2\pi y_j)] \ - \ 6 -\omega_0 \ = \ 0.
\end{equation}
In a forward collision, a particle with wave vector $y-x$ merges with a particle with wave vector $x$, resulting in a new particle with wave vector $y$. Following Remark \ref{Remarkdelta}, we could see that $y-x$ does not need to belong to $\mathbb{T}^d$. Indeed, there exists a vector $z\in\mathbb{Z}^d$ such that $y-x-z\in\mathbb{T}^d$. In this collision, the conservation of energy $\omega(y)=\omega(x)+\omega(y-x)$, describing by equation  \eqref{ForwardCollision}, needs to be satisfied. Therefore, given a particle with wave vector $x$, there maybe no wave vector $y$ such that the conservation of energy is guaranteed. In other words, there may be no $y$ such that $x$ is connected  to $y$ by a forward collision.

On the other hand, we say that the wave vector $x$ is connected to the wave vector $y=(y^1,y^2,y^3)\in\mathbb{T}^3$  by a {\bf backward collision}  if and only if
\begin{equation}\label{BackwardCollision}
\mathfrak{F}^b_x(y) \ := \ \sum_{j=1}^32[\cos(2\pi y_j) \  + \ \cos(2\pi(x_j-y_j)) \ - \ \cos(2\pi x_j)] \ - \ 6 -\omega_0 \ = \ 0.
\end{equation}
Different from  forward collisions, in a backward collision, a particle with wave vector $x$ is broken into two particles, one with wave vector $y$, and the other one with wave vector $x-y$. Again, in a backward collision, the conservation of energy $\omega(x)=\omega(y)+\omega(x-y)$ needs to be satisfied; and therefore, for a given wave vector $x$, it could happen that one cannot break $x$ into $y$ and $x-y$, such that the energy conservation \eqref{BackwardCollision} is satisfied. Again, following Remark \ref{Remarkdelta}, we could see that $x-y$ does not need to belong to $\mathbb{T}^d$. Indeed, there exists a vector $z\in\mathbb{Z}^d$ such that $x-y-z\in\mathbb{T}^d$.

Finally, we say that the wave vector $x$ is connected to the wave vector $y$ or the wave vector $y$ is connected to the wave vector $x$ by a {\bf central collision}  if and only if
\begin{equation}\label{CentralCollision}
\mathfrak{F}^c_x(y)\  = \ \mathfrak{F}^c_y(x) \ := \ \sum_{j=1}^32[\cos(2\pi y_j) \  + \ \cos(2\pi(x_j)) \ - \ \cos(2\pi (x_j+y_j))] \ - \ 6 -\omega_0 \ = \ 0.
\end{equation}
Similarly to the above types of collisions, in a central collision, we require that $\omega(x)+\omega(y)=\omega(x+y)$ and this conservation of energy is not always satisfied. Following Remark \ref{Remarkdelta}, we could see that $y+x$ does not need to belong to $\mathbb{T}^d$. Indeed, there exists a vector $z\in\mathbb{Z}^d$ such that $y+x-z\in\mathbb{T}^d$.

Note that if $y$ is connected to $x$ by a forward collision, then $x$ is connected to $y$ by a backward collision. Moreover, if $y$ is connected to $x$ by a central collision, then $x$ is connected to $y$ by a central collision and $x+y$ is connected to both $x$ and $y$ by backward collisions. We simply say that $x$ and $y$ are connected by one collision; or $x$ is connected to  $y$ and $y$ is connected to $x$ by one collision.

If a wave vector $k$ is not connected to any other wave vectors in  forward collisions, the second term in the collision operator $Q_c[f](k)$
$$\int_{\mathbb{T}^6}[\omega\omega_1\omega_2]^{-1}\delta(k_1-k-k_2)\delta(\omega_1-\omega-\omega_2)[f_2f-ff_1-f_1f_2]\mathrm{d}k_1\mathrm{d}k_2$$
vanishes, no matter how we choose the function $f$.

If a wave vector $k$ is not connected to any other wave vectors in backward collisions, the first term in the collision operator $Q_c[f](k)$
$$\int_{\mathbb{T}^6}[\omega\omega_1\omega_2]^{-1}\delta(k-k_1-k_2)\delta(\omega-\omega_1-\omega_2)[f_1f_2-ff_1-ff_2]\mathrm{d}k_1\mathrm{d}k_2$$
vanishes.

We define the set of all wave vectors $k$ such that $k$ is not connected to any other wave vectors to be {\bf the no-collision region $\mathfrak{I}$.} It is clear that $\mathfrak{F}^f_0(y)=\mathfrak{F}^c_0(y)=-\omega_0<0$ and $$\mathfrak{F}^b_0(y)=\sum_{j=1}^32[2\cos(2\pi y_j)-1] -6-\omega_0=\sum_{j=1}^32[2\cos(2\pi y_j)-2] -\omega_0\le \ -\omega_0<0,$$ for all wave vectors $y$. As a consequence, the origin belongs to $\mathfrak{I}$. Since $\mathfrak{F}^f_0(y), \mathfrak{F}_0^b(y), \mathfrak{F}^c_0(y)\le -\omega_0 <0$, there exists a ball $B(0,R):=\{x\in\mathbb{R}^3 \ \ |\ \  |x| <R\}$, $(R>0)$, such that $\mathfrak{F}^f_x(y), \mathfrak{F}_x^b(y), \mathfrak{F}^c_x(y)<0$, for all  $y\in\mathbb{T}^3$ and for all $x\in B(0,R)$. The ball $B(0,R)$ is therefore a subset of  the  no-collision  region $\mathfrak{I}$.

The condition $2<\omega_0< 3$ implies that the set $\mathbb{T}^3\backslash\mathfrak{I}$ is then not empty. 
For a vector $x\in \mathbb{T}^3\backslash\mathfrak{I}$, we define $\mathcal{S}^1(x)$ to be  the {\bf one-collision connection set of $x$}, containing  all wave vectors $y\in\mathbb{T}^3$ such that $y$ is connected to $x$ by a  collision.  By a recursive manner, we also define $\mathcal{S}^n(x)=\mathcal{S}^1(\mathcal{S}^{n-1}(x))$, the {\bf $n$-collision connection set of $x$}, for $n\ge 2, n\in\mathbb{N}$. This set  consists of all wave vectors connecting to $x$ by at most $n$ collisions. The union 
\begin{equation}\label{ConnectionSet}
\mathcal{S}(x)=\bigcup_{1\le n<\infty}\mathcal{S}^n(x)
\end{equation}
contains all wave vectors $y$  connecting to $x$ by a finite number of collisions.  We then call $\mathcal{S}(x)$ a {\bf finite collision connection set of $x$} or a {\bf collision invariant region}. 

Note that if $k\in\mathcal{S}(x)$ and $k$ is connected to $k+k'\in\mathcal{S}(x)$ by a forward collision,  then $k+k'$ is also connected with $k'$ by a backward collision, and hence $k'\in\mathcal{S}(x)$. 
\begin{proposition}[The effect of the collision operator on the no-collision region]\label{NoCollisionEffect}
Any smooth solution $f(t,k)$ of \eqref{PhononEqC}, is time invariant on the no-collision region $\mathfrak{I}$. In other words, $f(t,k)=f_0(k)$ for all $k\in\mathfrak{I}$.
\end{proposition}
\begin{proof}
Since  $k\in\mathfrak{I}$, the wave vector $k$ is not connected to any other wave vectors in any collisions, the collision operator $Q_c[f](k)$ vanishes, which implies $\partial_t f(t,k)=0$ for all $k\in\mathfrak{I}$. Therefore,  $f(t,k)=f_0(k)$ for all $k\in\mathfrak{I}$.
\end{proof}

\begin{proposition}[Decomposition into collisional invariant regions]
\label{Propo:ConnectionSet}
Let $x,y$ be two wave vectors in $\mathbb{T}^3\backslash\mathfrak{I}$, then either $\mathcal{S}(x)=\mathcal{S}(y)$ or $\mathcal{S}(x)\cap\mathcal{S}(y)=\emptyset$. In other words, either $x$ and $y$ are connected by a finite number of collisions ($\exists m>0$ such that $x\in \mathcal{S}^m(y)$) or they are totally disconnected ($\nexists m>0$ such that $x\in \mathcal{S}^m(y)$). 

As a consequence, there exists a subset $\mathfrak{V}$ of $\mathbb{T}^3\backslash\mathfrak{I}$ such that the torus $\mathbb{T}^3$ can be decomposed into disjoint  { collisional invariant regions}, as follows
\begin{equation}
\label{TorusDecomposition}
\mathbb{T}^3 \backslash\mathfrak{I} \ = \ \bigcup_{x\in \mathfrak{V}}\mathcal{S}(x),
\end{equation}
and $\mathcal{S}(x)\cap\mathcal{S}(y)=\emptyset$ for $x,y\in \mathfrak{V}$.
\end{proposition}
\begin{proof}
Let $x,y$ be two wave vectors in $\mathbb{T}^3\backslash\mathfrak{I}$ and suppose that $\mathcal{S}(x)\cap\mathcal{S}(y)\neq\emptyset$, we can therefore choose a wave vector $z$ belonging to both sets $\mathcal{S}(x)$ and  $\mathcal{S}(y)$, that means $z$ is connected to both wave vectors $x$ and $y$ by finite numbers of collisions. It follows that $z\in \mathcal{S}^n(x)$ and $z\in \mathcal{S}^m(y)$, for some positive integers $n$ and $m$. Since $z\in  \mathcal{S}^n(x)$, it is clear that $\mathcal{S}(z)\subset\mathcal{S}^{n+1}(x)$, and in general $\mathcal{S}^p(z)\subset\mathcal{S}^{n+p}(x)$ for all $p\in\mathbb{N}$. As a result, $\mathcal{S}(z)\subset\mathcal{S}(x)$. By a similar argument, it also follows that $\mathcal{S}(z)\subset\mathcal{S}(y)$. Now, let $\vartheta$ be an wave vector of $\mathcal{S}(y)\backslash \mathcal{S}(z)$. Being a wave vector of $\mathcal{S}(y)$, $\vartheta$ is connected to $y$ by a finite number $p\in\mathbb{N}$ of collisions. Since $z$ is connected to $y$ by $m$ collisions,  $\vartheta$ is connected to $z$ by at most $p+m$ collisions. In other words, $\vartheta\in \mathcal{S}^{p+m}(z)$; and hence, $\vartheta\in \mathcal{S}(z)$, contradicting the fact that $\vartheta \in \mathcal{S}(y)\backslash \mathcal{S}(z)$. This contradiction leads to $\mathcal{S}(y)\subset \mathcal{S}(z)$; however, as shown above $\mathcal{S}(z)\subset\mathcal{S}(y)$, it then follows  $\mathcal{S}(y)= \mathcal{S}(z)$. The same argument can also be  used to prove  $\mathcal{S}(x)= \mathcal{S}(z)$. We finally get $\mathcal{S}(y)= \mathcal{S}(x)$. 

The existence of $ \mathfrak{V}$ and the decomposition \eqref{TorusDecomposition} then follows straightforwardly. 
\end{proof}
\begin{remark}
The decomposition of the domain $\mathbb{T}^3$ in to several collisional invariant and no-collision regions is a very special and interesting feature of the specific form of the dispersion relation \eqref{DispersionRelation}. 

In the previous works, several other dispersion relations have been considered in many other contexts $\omega(k)=|k|$ for very low temperature bosons (see \cite{AlonsoGambaBinh,EscobedoBinh}), $\omega(k)=|k|^\gamma$, $(1<\gamma\le 2)$ for capillary waves (see \cite{nguyen2017quantum}),  $\omega(k)=\sqrt{c_1|k|^2+c_2|k|^4}$, $(0<c_1,0\le c_2)$ for bosons (see \cite{PomeauBinh,Binh1}) and the space of the frequency $k$ is $\mathbb{R}^d$. 
{In all of these cases, the division of the domain of wavenumbers into disjoint regions has never been observed due to the fact that the frequency space is $\mathbb{R}^d$ instead of $\mathbb{T}^d$. On the other hand, important results on 4-wave kinetic equations set the 
torus $\mathbb{T}^d$ have been recently obtained in \cite{escobedo2024entropy,germain2024stability,  lukkarinen2008anomalous}.}

Notice that in \cite{GambaSmithBinh}, the dispersion relation $\omega(k)=\sqrt{c_1+c_2|k|^2}$, $(0<c_1,c_2)$ for stratified flows in the ocean, has been considered. However, the resonance is broadened and the extended resonance manifold is then studied
$$k=k_1+k_2, \ \ \ |\omega(k) - \omega(k_1) - \omega(k_2)|\le \theta,  \ \ \ k, k_1, k_2\in\mathbb{R}^2,$$
for $\theta>0$, in stead of the exact resonance one
$$k=k_1+k_2, \ \ \ \omega(k) = \omega(k_1) + \omega(k_2),  \ \ \ k, k_1, k_2\in\mathbb{R}^3,$$
{due to the fact that the exact resonance configuration is no longer correct} (see \cite{smith2002generation}). Of course, in all resonance broadening cases, the decomposition of the full domain into local no-collision and collisional invariant regions does not exist. 

\end{remark}

\begin{proposition}\label{Propo:SxClosed}
The set $\mathcal{S}^n(x)$ is a closed subset of $\mathbb{T}^3$ for all $n\in\mathbb{N}\backslash\{0\}$.
\end{proposition}
\begin{proof}
We first observe that the set $\mathcal{S}^1(x)$ contains all wave vectors $y$ such that $x$ is connected to $y$ by either  a forward, a backward  or a central collision. By definition, the set of all $y$ such that $x$ is connected to $y$ by a forward collision is 
\begin{equation}\label{Propo:SxClosed:E1}
\mathcal{S}^1_f(x) \ = \ \left[\mathfrak{F}_x^f\right]^{-1}\left(\left\{0\right\}\right).
\end{equation}
Similarly, the sets of all $y$ such that $x$ is connected to $y$ by backward and central collisions are 
\begin{equation}\label{Propo:SxClosed:E2}
\mathcal{S}^1_b(x) \ = \ \left[\mathfrak{F}_x^b\right]^{-1}\left(\left\{0\right\}\right),
\end{equation}
and
\begin{equation}\label{Propo:SxClosed:E3}
\mathcal{S}^1_c(x) \ = \ \left[\mathfrak{F}_x^c\right]^{-1}\left(\left\{0\right\}\right).
\end{equation}
By the continuity of $\mathfrak{F}_x^f, \mathfrak{F}_x^b$ and $\mathfrak{F}_x^c$, the sets $\mathcal{S}^1_f(x)$, $\mathcal{S}^1_b(x)$ and $\mathcal{S}^1_c(x)$ are all closed. Since $\mathcal{S}^1(x)=\mathcal{S}^1_f(x)\cup \mathcal{S}^1_b(x)\cup \mathcal{S}^1_c(x)$, it is also a closed set. 

We now follow an induction argument in $n$. When $n=1$, it is clear from the above argument that $\mathcal{S}^1(x)$ is closed. Suppose that $\mathcal{S}^k(x)$ is closed, we will show that $\mathcal{S}^{k+1}(x)$ is also closed for all $k\ge 1$. To this end, let us suppose that $\{x_m\}_{m=1}^\infty$ is a sequence in $\mathcal{S}^{k+1}(x)$ and $\lim_{m\to\infty}x_m=x_*$. By the definition of the set $\mathcal{S}^{k+1}(x)$, there exists a sequence $\{y_m\}_{m=1}^\infty$ such that $y_m\in \mathcal{S}^k(x)$ and either $\mathfrak{F}^f_{y_m}(x_m)=0$,  $\mathfrak{F}^b_{y_m}(x_m)=0$ or $\mathfrak{F}^c_{y_m}(x_m)=0$. Without loss of generality, we can assume that there exist  subsequences $\{x_{m_q}\}_{q=1}^\infty$ and $\{y_{m_q}\}_{q=1}^\infty$ of $\{x_m\}_{m=1}^\infty$ and $\{y_m\}_{m=1}^\infty$ such that $\mathfrak{F}^f_{y_{m_q}}(x_{m_q})=0$.  Since the sequence $\{y_{m_q}\}_{q=1}^\infty$  is a subset of $\mathcal{S}^k(x)$, which is closed and hence compact, there exists a subset of $\{y_{m_q}\}_{q=1}^\infty$, still denoted by $\{y_{m_q}\}_{q=1}^\infty$, such that this sequence has a limit $y_*\in\mathcal{S}^k(x)$ as $q$ tends to infinity. By the continuity of $\mathfrak{F}^f_y(x)$ in both $x$ and $y$, $\lim_{q\to\infty}\mathfrak{F}^f{y_{m_q}}(x_{m_q})=\mathfrak{F}^f_{y_*}(x_*)$. That implies $\mathfrak{F}^f_{y_*}(x_*)=0$ and hence $x_*\in\mathcal{S}^{k+1}(x)$. We finally conclude that the set $\mathcal{S}^{k+1}(x)$ is closed. By induction $\mathcal{S}^n(x)$ is closed for all $n\in\mathbb{N}\backslash\{0\}$. 

\end{proof}

\begin{corollary}\label{Cor:SxMeasurable}
The set $\mathcal{S}(x)$ is Lebesgue measurable. 
\end{corollary}
\begin{proof}
The proof of this corollary follows directly from Proposition \ref{Propo:SxClosed} and the definition of $\mathcal{S}(x)$.
\end{proof}

\begin{remark}
The two sets $\mathcal{S}^1_f(x) $ and $\mathcal{S}^1_b(x) $ defined in \eqref{Propo:SxClosed:E1} and \eqref{Propo:SxClosed:E2} are indeed disjoint. This can be seen by a proof of contradiction. Suppose that $y$ is a common wave vector of both $\mathcal{S}^1_f(x) $ and $\mathcal{S}^1_b(x) $. This means
$$\sum_{i=1}^32[\cos(2\pi(y_i-x_i))+\cos(2\pi x_i) - \cos(2\pi y_i)]\  = \ 6+\omega_0,$$
and
$$\sum_{i=1}^32[\cos(2\pi(x_i-y_i))+\cos(2\pi y_i) - \cos(2\pi x_i)]\  = \ 6+\omega_0.$$
Taking the sum of the above two identities yields 
$$\sum_{i=1}^32\cos(2\pi(y_i-x_i))\  = \ 6+\omega_0.$$
The left hand side is smaller than or equal to $6$, while the right hand side is strictly greater than $6$ due to the fact that $\omega_0>0$. This leads to a contradiction; and thus, $\mathcal{S}^1_f(x) $ and $\mathcal{S}^1_b(x) $ are  disjoint. However, $\mathcal{S}^1_c(x) $ can have common wave vectors with both $\mathcal{S}^1_f(x) $ and $\mathcal{S}^1_b(x) $.
\end{remark}

\subsubsection{Continuity of  set index functionals}\label{Sec:lip}
In the study of  the wave kinetic equation, we frequently encounter integrals of the types
\begin{equation}\label{RestrictedIntegral1}
\int_{\mathbb{T}^3}\delta(\omega(x)-\omega(x-y)-\omega(y))f(y)\mathrm{d}y,
\end{equation}
\begin{equation}\label{RestrictedIntegral2}
\int_{\mathbb{T}^3}\delta(\omega(y)-\omega(y-x)-\omega(x))f(y)\mathrm{d}y,
\end{equation}
and
\begin{equation}\label{RestrictedIntegral3}
\int_{\mathbb{T}^3}\delta(\omega(x+y)-\omega(x)-\omega(y))f(y)\mathrm{d}y.
\end{equation}
Special cases of \eqref{RestrictedIntegral1}-\eqref{RestrictedIntegral2}-\eqref{RestrictedIntegral3} involve $f(y)=\chi_A(y)$, the characteristic function of a Lebesgue measurable set $A$. \begin{definition}[Index functionals of sets]\label{Def:Measure}
	Let $A$ be a Lebesgue measurable set, we define the following three functionals.
	\begin{itemize}
		\item[(I)] The ``forward collision'' index of the set $A$:
		\begin{equation}\label{Def:Measure:E1}
			\mu_1[A](x) \ := \ \int_{\mathbb{R}}\int_{\mathbb{T}^3}e^{it(\omega(x)-\omega(x-y)-\omega(y))}\chi_{A}(y)\mathrm{d}y\mathrm{d}t,
		\end{equation}
		where $\chi_A$ is the characteristic function of the set $A$. 
		\item[(II)] The ``backward collision'' index of the set $A$:
		\begin{equation}\label{Def:Measure:E1}
			\mu_2[A](x) \ := \ \int_{\mathbb{R}}\int_{\mathbb{T}^3}e^{it(\omega(y)-\omega(y-x)-\omega(x))}\chi_{A}(y)\mathrm{d}y\mathrm{d}t,
		\end{equation}
		where $\chi_A$ is the characteristic function of the set $A$. 
		\item[(III)] The ``central collision'' index of the set $A$:
		\begin{equation}\label{Def:Measure:E1}
			\mu_3[A](x) \ := \ \int_{\mathbb{R}}\int_{\mathbb{T}^3}e^{it(\omega(x+y)-\omega(x)-\omega(y))}\chi_{A}(y)\mathrm{d}y\mathrm{d}t,
		\end{equation}
		where $\chi_A$ is the characteristic function of the set $A$. 
	\end{itemize}
\end{definition} For the sake of simplicity, in this section, we denote $\mu_1(\mathbb{T}^3)$, $\mu_2(\mathbb{T}^3)$ and $\mu_3(\mathbb{T}^3)$  by $F(x)$,  $G(x)$ and $H(x)$. 
\begin{proposition} \label{Propo:Lip}  The  functions $F(x)$,  $G(x)$ and $H(x)$ are   {continuous on the set $$\mathfrak{S}=\Big\{x=(x^1,x^2,x^3)\in \mathbb{T}^3 \mbox{ in which } x^i\neq \pm\frac12,0, \mbox{  for all } i=1,2,3\Big\}.$$}
\end{proposition}

\begin{proof}
Notice that 
	\begin{equation}
		\label{Propo:Measure:E6}
		\omega(x)-\omega(x-y)-\omega(y) \ = \ -\omega_0 \ - \ 6 \ + \ \sum_{i=1}^32[\cos(2\pi x^i-2\pi y^i) \ + \ \cos(2\pi y^i) \ - \ \cos(2\pi x^i)] ,
	\end{equation}
	where $x=(x^1,x^2,x^3)$, $y=(y^1,y^2,y^3)$. \\
	We will need to bound
	
	\begin{equation}
		\label{Propo:Measure:E7}\begin{aligned}
			\mathcal{J} \ = \ & \int_{\mathbb{T}^3}e^{it(\sum_{i=1}^32[\cos(2\pi x^i-2\pi y^i) \ + \ \cos(2\pi y^i)])} \mathrm{d}y\\
			= \ & \int_{\mathbb{T}}e^{it2[\cos(2\pi x^1-2\pi y^1) \ + \ \cos(2\pi y^1)]} \mathrm{d}y^1\int_{\mathbb{T}} e^{it2[\cos(2\pi x^2-2\pi y^2) \ + \ \cos(2\pi y^2)]} \mathrm{d}y^2\times\\
			\ &\times\int_{\mathbb{T}}e^{it2[\cos(2\pi x^3-2\pi y^3) \ + \ \cos(2\pi y^3)]} \mathrm{d}y^3\\
			= \ & \mathcal{J}_1\times \mathcal{J}_2\times \mathcal{J}_3
	\end{aligned}\end{equation}
	which is a product of three oscillation integrals with  phases $t\Phi_i(y)$, where $\Phi_i(y)=2[\cos(2\pi x^i-2\pi y^i) \ + \ \cos(2\pi y^i)]$, $i=1,2,3.$

	To estimate \eqref{Propo:Measure:E7}, we will use the method of stationary phase. Let us  point out that in \cite{germain2017optimal}, the authors use different kinds of techniques,  to estimate integrals of similar types but for different classes of dispersion relations. Notice that $\partial_{y^i} \Phi_i(y^i)=-4\pi\sin(2\pi y^i-2\pi x^i)-4\pi\sin(2\pi y^i)=0$ when $y^i=\frac{x^i}{2}$,    $y^i=\frac12+\frac{x^i}{2}$,   or $x^i=\pm\frac12$. Observe that when $y^i=\frac{x^i}{2}$,    $y^i=\frac12+\frac{x^i}{2}$, we have $|\partial_{y^iy^i}\Phi_i(y^i)| =8\pi^2|\cos(2\pi y^i-2\pi x^i)+\cos(2\pi y^i)|=16\pi^2|\cos(\pi x^i)|=8\pi^2|1+e^{i2\pi x^i}|.$  
	
 {We observe that all $x^i$, $i=1,2,3$, need to be different from $\pm\frac12$. This fact could be seen by a proof of contradiction, in which we suppose that $x^1$ is equal to $\frac12$ or $-\frac12$ as follows. By Proposition \ref{Propo:SxClosed},  $\mathcal{S}(x)$  is non-empty,} then either 
	$$
	0=\omega(x)-\omega(x-y)-\omega(y) \ = \ -\omega_0 \ - \ 6 \ + \ \sum_{i=1}^32[\cos(2\pi x^i-2\pi y^i) \ + \ \cos(2\pi y^i) \ - \ \cos(2\pi x^i)] ,
	$$  $$
	0=\omega(x+y)-\omega(x)-\omega(y) \ = \ -\omega_0 \ - \ 6 \ + \ \sum_{i=1}^32[\cos(2\pi x^i) \ + \ \cos(2\pi y^i) \ - \ \cos(2\pi x^i+2\pi y^i)],
	$$
	or
	$$
	0=\omega(y)-\omega(x)-\omega(y-x) \ = \ -\omega_0 \ - \ 6 \ + \ \sum_{i=1}^32[\cos(2\pi x^i) \ + \ \cos(2\pi y^i-2\pi x^i) \ - \ \cos(2\pi y^i)],
	$$ has to have a solution. Let us consider the first equation. Plugging  the values $\pm\frac{1}{2}$ of $x^1$  into the equation yields $$
	\omega_0 \ + \ 4 \ =\sum_{i=2}^32[\cos(2\pi x^i-2\pi y^i) \ + \ \cos(2\pi y^i) \ - \ \cos(2\pi x^i)],
	$$ which has no solutions since $\omega_0+4>6$ and $[\cos(2\pi \alpha-2\pi \beta) \ + \ \cos(2\pi \beta) \ - \ \cos(2\pi \alpha)]\le \frac32$ for all $\alpha,\beta\in\mathbb{T}$. Now, we consider the second equation, and plug the values $\pm\frac{1}{2}$ of  $x^1$ into the equation to get $$
	\omega_0 \ + \ 8 \  - \  4 \cos(2\pi y^1) =\sum_{i=2}^32[\cos(2\pi x^i) \ + \ \cos(2\pi y^i) \ - \ \cos(2\pi x^i+2\pi y^i)],
	$$
	which also has no solution since $\omega_0 \ + \ 8 \  - \  4 \cos(2\pi y^1)>6$ and  $[\cos(2\pi \alpha) \ + \ \cos(2\pi \beta) \ - \ \cos(2\pi \alpha+2\pi \beta)]\le \frac32$ for all $\alpha,\beta\in\mathbb{T}$.  Finally, in the last case, the same argument gives
	$$
	\omega_0 \ + \ 8 \  + \  4 \cos(2\pi y^1) =\sum_{i=2}^32[\cos(2\pi x^i) \ + \ \cos(2\pi y^i-2\pi x^i) \ - \ \cos(2\pi y^i)],
	$$
	which again has no solution.
	
	Since $x^i$ is different from  $\pm\frac12$, it is clear that $\partial_{y^i} \Phi_i(y^i)=-4\pi\sin(2\pi y^i-2\pi x^i)-4\pi\sin(2\pi y^i)=0$ when $y^i=\frac{x^i}{2}$ and   $y^i=\frac12+\frac{x^i}{2}$.   By the method of stationary phase

	\begin{equation}
		\label{Propo:Measure:E10}\mathcal{J}_i\lesssim \frac{1}{\langle t\rangle^\frac12\sqrt{|1+e^{i2\pi x^i}|}},
	\end{equation}
	when $x^i$ is different from  $\pm\frac12$.
	
	Multiplying all inequalities \eqref{Propo:Measure:E10} for $i=1,2,3$ yields
	\begin{equation}
		\label{Propo:Measure:E11}\mathcal{J}\lesssim \frac{1}{\langle t\rangle^\frac32\sqrt{|1+e^{i2\pi x^1}||1+e^{i2\pi x^2}||1+e^{i2\pi x^3}|}}.
	\end{equation}

	Let $x$ be a point in $\mathfrak{S}$ and a sequence $\{x_n\}_{n=1}^\infty\subset\mathfrak{S}$ such that $\lim_{n\to\infty}x_n=x$. Since the set $\mathbb{T}^3\backslash\mathfrak{S}$ is closed, without loss of generality, we suppose that there exists a ball $B(x,r)$ with radius $r$ and centered at $x$ such that $B(x,r)\cap (\mathbb{T}^3\backslash\mathfrak{S})=\emptyset$ and then $\{x_n\}_{n=1}^\infty\subset B (x,r)$. 
	From the  assumption $B(x,r)\cap (\mathbb{T}^3\backslash\mathfrak{S})=\emptyset$, it follows  
	\begin{equation}\label{Propo:Lip:E1}
		\begin{aligned}
			\left|\int_{\mathbb{T}^3}e^{it(\omega(x)-\omega(x-y)-\omega(y))}\mathrm{d}y\right| \ \lesssim & \ \frac{1}{\langle t\rangle^\frac32 \sqrt{|1+e^{2\pi x^1}| |1+e^{2\pi x^2}| |1+e^{2\pi x^3}|}} \ \lesssim \ 1.
		\end{aligned}
	\end{equation}
	By the Lebesgue dominated convergence theorem, $\lim_{n\to\infty}F(x_n)=F(x)$ and the function $F$ is then continuous on $\mathfrak{S}$.
 By the same argument, $G,H$ are also  continuous.
\end{proof}
\begin{corollary}\label{Coll:Edges}
	The  edges, i.e. the set $\mathbb{T}^3\backslash\mathfrak{S}$ of all wave vectors $y=(y^1,y^2,y^3)$ in which there is an index $i\in\{1,2,3\}$ such that $y^i= \pm\frac12$ or $0$, is a subset of the no-collision region $\mathfrak{I}$.
\end{corollary}
\begin{proof}
	The corollary follows directly from the proof of Proposition \ref{Propo:Lip}.
\end{proof}

\subsubsection{Restrictions on $\mathcal{S}(x)$}\label{Sec:Restriction} 

\begin{proposition}\label{Lemma:RestrictionSx}
	Given any function $f\in \mathbb{L}^1(\mathbb{T}^3)$ and a collisional invariant region $\mathcal{S}(x)$. Define restriction of $f$ on $\mathcal{S}(x)$ as follows
	\begin{equation}
		\label{Lemma:RestrictionSx:1}
		f_{|_{\mathcal{S}(x)}}(y) \ =  \ f(y) \mbox{ if } y\in\mathcal{S}(x) \mbox{ and }  f_{|_{\mathcal{S}(x)}}(y) \ =  \ 0 \mbox{ if } y\in\mathbb{T}^3\backslash\mathcal{S}(x).
	\end{equation}
	Then, in the distributional sense, we have \begin{equation}\label{Lemma:RestrictionSx:2}
		\int_{\mathbb{T}^3}\delta(\omega(x)-\omega(x-y)-\omega(y))f(y)\mathrm{d}y \ = \ \int_{\mathbb{T}^3}\delta(\omega(x)-\omega(x-y)-\omega(y))f_{|_{\mathcal{S}(x)}}(y)\mathrm{d}y,
	\end{equation}
	\begin{equation}\label{Lemma:RestrictionSx:3}
		\int_{\mathbb{T}^3}\delta(\omega(y)-\omega(y-x)-\omega(x))f(y)\mathrm{d}y \ = \ \int_{\mathbb{T}^3}\delta(\omega(y)-\omega(y-x)-\omega(x))f_{|_{\mathcal{S}(x)}}(y)\mathrm{d}y,
	\end{equation}
	and
	\begin{equation}\label{Lemma:RestrictionSx:4}
		\int_{\mathbb{T}^3}\delta(\omega(x+y)-\omega(x)-\omega(y))f(y)\mathrm{d}y \ = \ \int_{\mathbb{T}^3}\delta(\omega(x+y)-\omega(x)-\omega(y))f_{|_{\mathcal{S}(x)}}(y)\mathrm{d}y.
	\end{equation}
\end{proposition}
\begin{proof}  {We only prove \eqref{Lemma:RestrictionSx:2}, as the proofs of \eqref{Lemma:RestrictionSx:3}-\eqref{Lemma:RestrictionSx:4} follow by the same argument.
 For a fixed value of $x$, we denote by $A_\theta$ with $\theta>0$ the set of all $z$ in $A$  such that \begin{equation}
		\label{Propo:Measure:E17}|\omega(x)-\omega(z)-\omega(x-z)|>\theta>0\end{equation}
	for all $z$ in $A$. }
	
	Let us introduce the following approximation 
	\begin{equation}
		\label{Propo:Measure:E3}
		\int_\mathbb{R}\int_{\mathbb{T}^3}e^{it(\omega(x)-\omega(x-y)-\omega(y))-\epsilon^2t^2}\chi_{A_\theta}(y) f(t)\mathrm{d}y\mathrm{d}t.
	\end{equation}
	Integrating in $t$, we obtain from \eqref{Propo:Measure:E3} 
	\begin{equation}
		\label{Propo:Measure:E4}
		\frac{C}{\epsilon}\int_{\mathbb{T}^3}e^{-\frac{\pi(\omega(x)-\omega(x-y)-\omega(y)^2}{\epsilon^2}}\chi_{A_\theta}(y)f(y) \mathrm{d}y,
	\end{equation}
	for some universal positive constant $C$.
	
	Combining  \eqref{Propo:Measure:E17} with the approximation \eqref{Propo:Measure:E3}, we find
	
	\begin{equation*}
		\begin{aligned}
		&	\int_\mathbb{R}\int_{\mathbb{T}^3}e^{it(\omega(x)-\omega(x-z)-\omega(z))-\epsilon^2t^2}\chi_{A_\theta}(z)f(z)  \mathrm{d}y\mathrm{d}t\\ \ = & \ 
			\frac{C}{\epsilon}\int_{\mathbb{T}^3}e^{-\frac{\pi(\omega(x)-\omega(x-z)-\omega(z))^2}{\epsilon^2}}\chi_{A_\theta}(z)f(z)  \mathrm{d}z\\ 
			\ \lesssim & \ 
			\frac{1}{\epsilon}\int_{\mathbb{T}^3}e^{-\frac{\pi\theta^2}{\epsilon^2}}\chi_{A_\theta}(z)f(z)  \mathrm{d}z. 
		\end{aligned}
	\end{equation*}
	Using the fact that $\chi_{A_\theta}$ is a subset of $\mathbb{T}^3$, we deduce
	\begin{equation}
		\label{Propo:Measure:E18}
		\begin{aligned}
			\int_\mathbb{R}\int_{\mathbb{T}^3}e^{it(\omega(x)-\omega(x-z)-\omega(z))-\epsilon^2t^2}\chi_{A_\theta}(z)f(z)  \mathrm{d}z\mathrm{d}t 
			\ \lesssim & \ 
			\frac{e^{-\frac{\pi\theta^2}{\epsilon^2}}}{\epsilon} \to 0 \mbox{ as } \epsilon \to 0.
		\end{aligned}
	\end{equation}
Let $\varphi(x)$ be a test function in $C^\infty(\mathbb{T}^d)$. 
	Again, the same stationary phase argument used in Proposition \ref{Propo:Lip} can be applied to show that 
		\begin{equation}
		\label{Propo:Measure:E19}
		\begin{aligned}
			\Big|\int_\mathbb{R}\int_{\mathbb{T}^3}e^{it(\omega(x)-\omega(x-z)-\omega(z))-\epsilon^2t^2}\varphi(x)  \mathrm{d}z\mathrm{d}t \Big|
			\ \lesssim & \ 1,
		\end{aligned}
	\end{equation}
	uniformly in $\epsilon$. By the Lebesgue dominated convergence theorem, we find

	\begin{equation}
		\label{Propo:Measure:E20}		\begin{aligned}
&		\int_\mathbb{R}\int_{\mathbb{T}^6}e^{it(\omega(x)-\omega(x-z)-\omega(z))}\chi_A(z)\varphi(x) \mathrm{d}z\mathrm{d}x\mathrm{d}t\\ =\ & \lim_{\theta\to0}\lim_{\epsilon\to0}\int_\mathbb{R}\int_{\mathbb{T}^6}e^{it(\omega(x)-\omega(x-z)-\omega(z))-\epsilon^2t^2}\chi_{A_\theta}(z)f(z)\varphi(x) \mathrm{d}z\mathrm{d}x\mathrm{d}t =0.		\end{aligned}
	\end{equation}

\end{proof}

\subsubsection{Weak formulation, local conservation of energy on collisional invariant regions}\label{Sect:WeakClassical}

\begin{lemma}\label{Lemma:WeakFormulationClassical}
For any smooth function $f(k)$, there holds 
\begin{eqnarray*}
\int_{\mathbb{T}^3}Q_c[f](k)\varphi(k)\mathrm{d}k
& = & \iiint_{\mathbb{T}^9}[\omega\omega_1\omega_2]^{-1}\delta(k-k_1-k_2)\delta(\omega-\omega_1-\omega_2)\times\\
&& \times [f_1f_2-ff_1-ff_2]
\Big( \varphi(k)-\varphi(k_1)-\varphi(k_2) \Big) \; \mathrm{d}k\mathrm{d}k_1\mathrm{d}k_2
\end{eqnarray*}
for any smooth test function $\varphi$. 

If $\varphi$ is supported in a collisional invariant region $\mathcal{S}(x)$, then, we also have
\begin{eqnarray*}
\int_{\mathbb{T}^3}Q_c[f](k)\varphi(k)\mathrm{d}k
& = & \iiint_{\mathcal{S}(x)\times\mathcal{S}(x)\times\mathcal{S}(x)}[\omega\omega_1\omega_2]^{-1}\delta(k-k_1-k_2)\delta(\omega-\omega_1-\omega_2)\times\\
&& \times [f_1f_2-ff_1-ff_2]
\Big( \varphi(k)-\varphi(k_1)-\varphi(k_2) \Big) \; \mathrm{d}k\mathrm{d}k_1\mathrm{d}k_2.
\end{eqnarray*}
\end{lemma}
\begin{proof}
We have 
$$\begin{aligned}
& \ \int_{\mathbb{T}^3}Q[f](k)\varphi(k)\mathrm{d}k \ = \\ 
& = \ \int_{\mathbb{T}^9}[\omega\omega_1\omega_2]^{-1}\delta(k-k_1-k_2)\delta(\omega-\omega_1-\omega_2)[f_1f_2-ff_1-ff_2]\varphi(k)\mathrm{d}k\mathrm{d}k_1\mathrm{d}k_2\\
& - \ \int_{\mathbb{T}^9}[\omega\omega_1\omega_2]^{-1}\delta(k_1-k-k_2)\delta(\omega_1-\omega-\omega_2)[f_2f-ff_1-f_1f_2]\varphi(k)\mathrm{d}k\mathrm{d}k_1\mathrm{d}k_2\\
& - \ \int_{\mathbb{T}^9}[\omega\omega_1\omega_2]^{-1}\delta(k_1-k-k_2)\delta(\omega_1-\omega-\omega_2)[f_2f-ff_1-f_1f_2]\varphi(k)\mathrm{d}k\mathrm{d}k_1\mathrm{d}k_2,
\end{aligned}
$$by switching the variables $k\leftrightarrow k_1$ and $k\leftrightarrow k_2$ in the second and third integrals, respectively, the first identity follows. The second identity follows straightforwardly from Corollary \ref{Lemma:RestrictionSx} and the first identity.  
\end{proof}

As a consequence, we obtain the following  corollary. 

\begin{corollary}[Conservation of   energy on  collisional invariant regions]\label{Cor-ConservatioMomentumC} Smooth solutions $f(t,k)$ of \eqref{PhononEqC}, with initial data $f(0,k) =f_0(k)$, satisfy 
\begin{eqnarray}\label{Coro:ConservatioMomentumC} 
\int_{\mathcal{S}(x)}f(t,k)\omega(k)\mathrm{d}k&=&\int_{\mathcal{S}(x)}f_0(k)\omega(k)\mathrm{d}k. 
\end{eqnarray}
for all $t\ge 0$ and for all $x\in \mathfrak{V}$, defined in Proposition \ref{Propo:ConnectionSet}. 
\end{corollary}
\begin{proof} This follows from Lemma \ref{Lemma:WeakFormulationClassical} by taking $\varphi(k) = \omega(k)$ with $k=(k^1,k^2,k^3)$.
 \end{proof}
\subsubsection{Local equilibria on collisional invariant regions}\label{Sec:LocalEqui}
In this section, we establish the form of local equilibria on collisional invariant regions. The key different between these local equilibria and the equilibria of classical kinetic equations is that these equilibria are only defined locally on collisional invariant regions. This is a very special feature of the 3-wave kinetic equation.

\begin{lemma}[$C^2$-collisional invariants]\label{Lemma:CollsionInvariant}
Let $\psi\in C^2(\mathcal{S}(x))$ be a  collisional invariant on the collisional invariant region $\mathcal{S}(x)$, in the following sense. For any wave vectors $k, k_1, k_2 \in \mathcal{S}(x)$, $$k=k_1+k_2+z, \mbox{ for some }z\in\mathbb{Z}^d,\ \ \ \ \omega(k)=\omega(k_1)+\omega(k_2),$$ we have $$\psi(k)=\psi(k_1)+\psi(k_2).$$ Then there exist a constant $a_x\in\mathbb{R}$, such that $$\psi(k)=a_x \omega(k).$$
\end{lemma}
\begin{proof}
Let us first prove that for $k\in\mathcal{S}(x)$, the partial derivatives $\partial_{k^j}\psi(k)$, with $k=(k^1,k^2,k^3)$, are well-defined. Without loss of generality, we only prove that the partial derivative with respect to the first component $\partial_{k^1}\psi(k)$ is well-defined. Since $k\in\mathcal{S}(x)$, there are two wave vectors $k_1,k_2$ such that either $k=k_1+k_2$ and $\omega(k)=\omega(k_1)+\omega(k_2)$; or $k+k_1=k_2$ and $\omega(k)+\omega(k_1)=\omega(k_2)$.

{\bf Case 1: $k=k_1+k_2$ and $\omega(k)=\omega(k_1)+\omega(k_2)$.} Since $\psi\in C^2(\mathbb{T}^3)$, in order to show that $\partial_{k^1}\psi(k)$ is well-defined at $k^1\in\mathbb{T}$,  we only have to prove that there exists $\epsilon>0$ such that  for each $\bar{k}^1\in (k^1-\epsilon,k^1+\epsilon)$ there are $\bar{k}^2,\bar{k}^3\in\mathbb{T}^3$, $\bar{k}=(\bar{k}^1,\bar{k}^2,\bar{k}^3)\in\mathcal{S}(x)$. For any $x,y\in\mathbb{T}$, define $$F(x,y)=\cos(2\pi(x+y))-\cos(2\pi x)-\cos(2\pi y).$$ Since $k=(k^1,k^2,k^3)=k_1+k_2=(k^1_1,k^2_1,k^3_1)+(k^1_2,k^2_2,k^3_2)$, we then have
$$F(k_1^1,k_2^1)+F(k_1^2,k_2^2)+F(k_1^3,k_2^3)=-\omega_0/2-3.$$
Now, we develop 
$$
\begin{aligned}
F(x,y)+1\ = & \ -\cos(2\pi x)-\cos(2\pi y)+1+\cos(2\pi(x+y))\\
 = & \ 2\cos\left(\pi({x+y})\right)\left[-\cos\left(\pi({x-y})\right)+\cos\left(\pi({x+y})\right)\right]\\
  = & \ -4\cos\left(\pi({x+y})\right)\sin\left(\pi x\right)\sin\left(\pi y\right) \ \le 4.
\end{aligned}.$$
Hence $\max_{x,y\in\mathbb{T}}F(x,y)=3$ when $(x,y)=\left(\frac12,-\frac12\right)=\left(-\frac12,\frac12\right).$ We observe that the sum $F(k_1^2,k_2^2)+F(k_1^3,k_2^3)$ must be strictly smaller than $6$; otherwise, $F(k_1^1,k_2^1)=-\omega_0/2-9<-9$, which is a contradiction. 

Since $F(k_1^2,k_2^2)+F(k_1^3,k_2^3)<6$, then for any $\delta$ small, either positive or negative, there exist $\delta_1, \delta_2$, either positive or negative, such that 
$$F(k_1^1+\delta,k_2^1)+F(k_1^2+\delta_1,k_2^2)+F(k_1^3+\delta_2,k_2^3)=-\omega_0/2-3,$$
due to the continuity of $F$. If  $\bar{k}^1=k^1+\delta$, then we choose $\bar{k}^2=k^1+\delta_1$ and $\bar{k}^3=k^3+\delta_2$.

{\bf Case 2: $k+k_1=k_2$ and $\omega(k)+\omega(k_1)=\omega(k_2)$.} Similar as Case 1, we only need to  show that, for each $k^1\in\mathbb{T}$, there exists $\epsilon>0$ such that  for each $\bar{k}^1\in (k^1-\epsilon,k^1+\epsilon)$ there are $\bar{k}^2,\bar{k}^3\in\mathbb{T}^3$, $\bar{k}=(\bar{k}^1,\bar{k}^2,\bar{k}^3)\in\mathcal{S}(x)$. Since $k_2=(k_2^1,k_2^2,k_2^3)=k_1+k=(k^1_1,k^2_1,k^3_1)+(k^1,k^2,k^3)$, we then have
$$F(k_1^1,k^1)+F(k_1^2,k^2)+F(k_1^3,k^3)=-\omega_0/2-3.$$

Since $F(k_1^2,k^2)+F(k_1^3,k^3)<6$, then for any $\delta$ small, either positive or negative, there exist $\delta_1, \delta_2$, either positive or negative, such that 
$$F(k_1^1,k^1+\delta)+F(k_1^2,k^2+\delta_1)+F(k_1^3,k^3+\delta_2)=-\omega_0/2-3,$$
due to the continuity of $F$. If $\bar{k}^1=k^1+\delta$, then we choose  $\bar{k}^2=k^1+\delta_1$ and $\bar{k}^3=k^3+\delta_2$.

Since on $\mathcal{S}(x)$, $\psi(k)$ is a function of $\omega(k)$ and $k$, there exists a twice differentiable continuous function $\phi\in C^2(\mathbb{R}_+\times\mathbb{T}^3)$ such that $\psi(k)=\varphi(\omega(k),k)$. 


For $k\in\mathcal{S}(x)$, there exist two wave vectors $k_1,k_2\in\mathbb{T}^3$, such that either $k=k_1+k_2$ and $\omega(k)=\omega(k_1)+\omega(k_2)$, or $k+k_1=k_2$ and $\omega(k)+\omega(k_1)=\omega(k_2)$. We assume that $k=k_1+k_2$ and $\omega(k)=\omega(k_1)+\omega(k_2)$, $k_1, k_2\in\mathbb{T}^3$, the other case can be consider with exactly the same argument. As we observe before, $k_1,k_2$ also belong to $\mathcal{S}(x)$ due to the fact that $k$ is connected to both $k_1,k_2$ by one-collisions. We have
$$\psi(k_1)+\psi(k_2)=\psi(k)=\varphi(\omega(k),k)=\varphi(\omega(k_1)+\omega(k_2),k_1+k_2).$$
We now follow the strategy of \cite{cercignani1999relativistic} and \cite{Spohn:TPB:2006}.
Differentiating the above identity with respect to $k^j_1$ and $k^j_2$ yields
\begin{equation*}
\begin{aligned}
\partial_{k^j_1}\psi(k_1) \ = & \ \partial_r\varphi(\omega(k),k)\partial_{k^j_1}\omega(k_1)+\partial_{k^j_1}\varphi(\omega(k),k),\\
\partial_{k^j_2}\psi(k_2) \ = & \ \partial_r\varphi(\omega(k),k)\partial_{k^j_2}\omega(k_2)+\partial_{k^j_2}\varphi(\omega(k),k).
\end{aligned}
\end{equation*}
Letting $i\in\{1,2,3\}$ be a different index, we manipulate the above identity as
\begin{equation*}
\begin{aligned}
 \ & \ (\partial_{k^j_1}\psi(k_1)-\partial_{k^j_2}\psi(k_2))(\partial_{k^i_1}\omega(k_1)-\partial_{k^i_2}\omega(k_2))\\
 \ = & \  (\partial_{k^i_1}\psi(k_1)-\partial_{k^i_2}\psi(k_2))(\partial_{k^j_1}\omega(k_1)-\partial_{k^j_2}\omega(k_2)).
\end{aligned}
\end{equation*}
We  differentiate the above identity in $k_1$, with $l$ being an index in  $\{1,2,3\}$ 
\begin{equation*}
\begin{aligned}
 \ & \ \partial_{k^j_1}\partial_{k^l_1}\psi(k_1)(\partial_{k^i_1}\omega(k_1)-\partial_{k^i_2}\omega(k_2))+(\partial_{k^j_1}\psi(k_1)-\partial_{k^j_2}\psi(k_2))\partial_{k^i_1}\partial_{k^l_1}\omega(k_1)\\
 \ = & \  \partial_{k^i_1}\partial_{k^l_1}\psi(k_1)(\partial_{k^j_1}\omega(k_1)-\partial_{k^j_2}\omega(k_2))+(\partial_{k^i_1}\psi(k_1)-\partial_{k^i_2}\psi(k_2))\partial_{k^j_1}\partial_{k^l_1}\omega(k_1),
\end{aligned}
\end{equation*}
and now  in $k_2$, with $h$ being an index in  $\{1,2,3\}$ 
\begin{equation*}
\begin{aligned}
 \ & \ \partial_{k^j_1}\partial_{k^l_1}\psi(k_1)\partial_{k^i_2}\partial_{k^h_2}\omega(k_2)+\partial_{k^j_2}\partial_{k^h_2}\psi(k_2)\partial_{k^i_1}\partial_{k^l_1}\omega(k_1)\\
 \ = & \  \partial_{k^i_1}\partial_{k^l_1}\psi(k_1)\partial_{k^j_2}\partial_{k^h_2}\omega(k_2)+\partial_{k^i_2}\partial_{k^h_2}\psi(k_2)\partial_{k^j_1}\partial_{k^l_1}\omega(k_1).
\end{aligned}
\end{equation*}
A particular case of the above identity is the following
\begin{equation*}
\begin{aligned}
 \ & \ \partial_{k^i_1}^2\psi(k_1)\partial_{k^j_2}^2\omega(k_2)\ = \ \partial_{k^j_1}^2\psi(k_1)\partial_{k^i_2}^2\omega(k_2),
\end{aligned}
\end{equation*}
which implies
$$\partial_{k^i_1}^2\psi(k_1)\cos(k_2^j) \ = \ \partial_{k^i_2}^2\psi(k_1)\cos(k_1^j),$$
for any $k_1,k_3\in\mathcal{S}(x)$, and $k_1,k_2$ are connected to $k_1+k_2$ by one collision. 

Hence $\psi(k)=a_x\omega(k) + b_x\cdot k + c_x$, with $a_x,c_x\in\mathbb{R}$, $b_x\in \mathbb{R}^3$ for any $k\in\mathcal{S}(x)$. By the fact  $\psi(k)=\psi(k_1)+\psi(k_2)$ whenever $k$ is connected to $k_1,k_2$ by one-collisions, it is straightforward that $c_x=b_x=0$. 
\end{proof}

\begin{proposition}[$L^1$-collisional invariants]\label{Propo:CollsionInvariant}
Let $\psi\in L^1(\mathcal{S}(x))$ be a  collisional invariant on the collisional invariant region $\mathcal{S}(x)$, in the following sense. For any $k\in \mathcal{S}(x)$, such that $$k=k_1+k_2,\mbox{ for some }z\in\mathbb{Z}^d,\ \ \ \ \omega(k)=\omega(k_1)+\omega(k_2),$$ we have $$\psi(k)=\psi(k_1)+\psi(k_2).$$ Then there exist a constant $a_x\in\mathbb{R}$, such that $$\psi(k)=a_x \omega(k).$$
\end{proposition}
\begin{proof}
For any function $\phi\in C^\infty(\mathbb{T}^3)$, we define the standard mollifier  $\phi_\delta(k)=\delta^{-3}\phi\left(\frac{k}{\delta}\right)$  and the standard approximation $\psi_\delta = \psi*\phi_\delta$ with $\delta>0$. It is then classical that $\lim_{\delta\to 0}\|\psi_\delta-\psi\|_{L^1(\mathcal{S}(x))}=0.$ 

Since $\psi(k)=\psi(k_1)+\psi(k_2)$, we also have $\psi_\delta(k)=\psi_\delta(k_1)+\psi_\delta(k_2)$. Lemma \ref{Lemma:CollsionInvariant} can be applied to $\psi_\delta$, yielding $\psi_\delta(k)=a_x^\delta \omega(k)  $ for some constant $a_x^\delta\in\mathbb{R}$. 
The conclusion of the Proposition then follows after passing $\delta$ to $0$, while taking into account the limit $\lim_{\delta\to 0}\|\psi_\delta-\psi\|_{L^1(\mathcal{S}(x))}=0.$  
\end{proof}

\begin{proposition}[Equilibria in  Collisional Invariant Regions]\label{Propo:ClassicalEqui}
Given a collisional invariant region $\mathcal{S}(x)$, a function $\mathcal{F}^c(k)\in C(\mathcal{S}(x))$ is said to be a local equilibrium of $Q_c$ on $\mathcal{S}(x)$ if and only if $Q_c[\mathcal{F}^c](k)=0$ and $\mathcal{F}^c(k)>0$ for all $k\in \mathcal{S}(x)$.

Let $E_x\in\mathbb{R}_+$  and assume 
\begin{equation}\label{Propo:ClassicalEqui:2}
\begin{aligned}
\int_{\mathcal{S}(x)}\frac{1}{a_x  }\mathrm{d}k \ = & \ E_x,
\end{aligned}
\end{equation}
with $a_x\in\mathbb{R}_+$; the local equilibrium on $\mathcal{S}(x)$ of $Q_c$ can be uniquely determined as
\begin{equation}\label{Propo:ClassicalEqui:3}
\mathcal{F}^c(k) \ = \ \frac{1}{a_x\omega(k)},
\end{equation}
subjected to the local energy  constraint  
\begin{equation}\label{Propo:ClassicalEqui:1}
\begin{aligned}
\int_{\mathcal{S}(x)}\mathcal{F}^c(k)\omega(k)\mathrm{d}k \ = & \ E_x. 
\end{aligned}
\end{equation}
\end{proposition}
\begin{proof}
Since $Q_c[\mathcal{F}^c](k)=0$ for all $k\in\mathcal{S}(x)$, using $\frac{1}{\mathcal{F}^c}$ as a test function, we obtain
\begin{equation}\label{Propo:ClassicalEqui:E1}
\begin{aligned}
0 \ = & \ \int_{\mathcal{S}(x)} Q_c[\mathcal{F}^c](k)\frac{1}{\mathcal{F}^c(k)}\mathrm{d}k\\ \ = & \ \int_{\mathcal{S}(x)\times \mathcal{S}(x)\times \mathcal{S}(x)}\delta(k-k_1-k_2)\delta(\omega-\omega_1-\omega_2)[\mathcal{F}^c_1\mathcal{F}^c_2-\mathcal{F}^c_1\mathcal{F}^c-\mathcal{F}^c_2\mathcal{F}^c]\times\\
& \ \times\left[\frac{1}{\mathcal{F}^c}-\frac{1}{\mathcal{F}^c_1}-\frac{1}{\mathcal{F}^c_2}\right]\mathrm{d}k\mathrm{d}k_1\mathrm{d}k_2\\
\ = & \ \int_{\mathcal{S}(x)\times \mathcal{S}(x)\times \mathcal{S}(x)}\delta(k-k_1-k_2)\delta(\omega-\omega_1-\omega_2)\mathcal{F}^c\mathcal{F}^c_1\mathcal{F}^c_2\left[\frac{1}{\mathcal{F}^c}-\frac{1}{\mathcal{F}^c_1}-\frac{1}{\mathcal{F}^c_2}\right]^2\mathrm{d}k\mathrm{d}k_1\mathrm{d}k_2,
\end{aligned}
\end{equation}
which implies $\frac{1}{\mathcal{F}^c}-\frac{1}{\mathcal{F}^c_1}-\frac{1}{\mathcal{F}^c_2}=0$ for all $k,k_1,k_2\in\mathcal{S}(x)$ satisfying $k=k_1+k_2$ in the periodic sense (i.e. there exists some $z\in\mathbb{Z}^d$ such that $k=k_1+k_2+z$) and $\omega=\omega_1+\omega_2$. Therefore $\frac{1}{\mathcal{F}^c}$ is a collisional invariant; and by Proposition \ref{Propo:CollsionInvariant}, $\mathcal{F}^c$ takes the form \eqref{Propo:ClassicalEqui:3}, given that the system \eqref{Propo:ClassicalEqui:2} has a unique solution $ a_x $.

\end{proof}

\subsubsection{Entropy formulation on the collisional invariant region $\mathcal{S}(x)$}
Let $f$ be a positive solution of \eqref{PhononEqC},  we define the local entropy on the collisional invariant region $\mathcal{S}(x)$ as follows
\begin{equation}
\label{EntropyC}
S_{c,\mathcal{S}(x)}[f] \ = \ \int_{\mathcal{S}(x)}s_c[f]\mathrm{d}k \ = \ \int_{\mathcal{S}(x)}\ln(f) \mathrm{d}k.
\end{equation}
In the sequel, we only consider the local entropy on one  collisional invariant region, then, for the sake of simplicity, we denote $S_{c,\mathcal{S}(x)}[f]$ by $S_{c}[f]$.

Now, we take the derivative in time of $S_c[f]$
\begin{equation}
\label{TimeDerivativeEntro1C}
\partial_tS_c[f] \ =  \  \int_{\mathcal{S}(x)}\frac{\partial_t f}{f}\mathrm{d}k.
\end{equation}
Replacing the quantity $\partial_t f$ in the above formulation by the right hand side of \eqref{PhononEqC} , we find
\begin{equation}
\label{TimeDerivativeEntro2C}
\begin{aligned}
 \  \partial_tS_c[f]
 \ = & \iiint_{\mathcal{S}(x)\times\mathcal{S}(x)\times\mathcal{S}(x)}[\omega\omega_1\omega_2]^{-1}\delta(k-k_1-k_2)\delta(\omega-\omega_1-\omega_2)\times\\
 & \times[f_1f_2-ff_1-ff_2]\frac{1}{f}\mathrm{d}k\mathrm{d}k_1\mathrm{d}k_2\\
& - \ 2\iiint_{\mathcal{S}(x)\times\mathcal{S}(x)\times\mathcal{S}(x)}[\omega\omega_1\omega_2]^{-1}\delta(k_1-k-k_2)\delta(\omega_1-\omega-\omega_2)\times\\
&\times[f_2f-ff_1-f_1f_2]\frac{1}{f}\mathrm{d}k\mathrm{d}k_1\mathrm{d}k_2.
\end{aligned}
\end{equation}
We now apply Lemma \ref{Lemma:WeakFormulationClassical} to the above identity to get
\begin{equation}
\label{TimeDerivativeEntro3C}
\begin{aligned}
 \  \partial_tS_c[f]
 \ = & \iiint_{\mathcal{S}(x)\times\mathcal{S}(x)\times\mathcal{S}(x)}[\omega\omega_1\omega_2]^{-1}\delta(k-k_1-k_2)\delta(\omega-\omega_1-\omega_2)[f_1f_2-ff_1-ff_2]\times\\
 & \times\left[\frac{1}{f_2}+\frac{1}{f_1}-\frac{1}{f}\right]\mathrm{d}k\mathrm{d}k_1\mathrm{d}k_2.
\end{aligned}
\end{equation}
By noting that $$f_1f_2-ff_1-ff_2=ff_1f_2\left[\frac{1}{f_1}+\frac{1}{f_2}-\frac{1}{f}\right],$$
we obtain from \eqref{TimeDerivativeEntro3C} the following entropy identity
\begin{equation}
\label{TimeDerivativeEntro4C}
\begin{aligned}
 \  \partial_tS_c[f]
 \ = & \ \int_{\mathcal{S}(x)}[\omega\omega_1\omega_2]^{-1}\delta(k-k_1-k_2)\delta(\omega-\omega_1-\omega_2)ff_1f_2\times\\
 \times  &\left[\frac{1}{f_1}+\frac{1}{f_2}-\frac{1}{f}\right]^2\mathrm{d}k\mathrm{d}k_1\mathrm{d}k_2\\
  =: & \ D_c[f].
\end{aligned}
\end{equation}
It is clear that the quantity $ D_c[f]$ is positive. Borrowing the idea of \cite{CraciunBinh,CraciunSmithBoldyrevBinh}, we now define the reciprocal, of $f$
\begin{equation}\label{gC} g=\frac{1}{f}.
\end{equation} As a consequence, the formula \eqref{TimeDerivativeEntro4C} can be expressed in the following form 
\begin{equation}
\label{TimeDerivativeEntro5C}
\begin{aligned}
 \  \partial_tS_c[f] \ = \ D_c[f] \ = \ \mathbb{D}_c[g]
 \ := & \iiint_{\mathcal{S}(x)\times\mathcal{S}(x)\times\mathcal{S}(x)}[\omega\omega_1\omega_2]^{-1}\delta(k-k_1-k_2)\delta(\omega-\omega_1-\omega_2)\times\\
 \times  &\frac{\left[g_1+g_2-g\right]^2}{gg_1g_2}\mathrm{d}k\mathrm{d}k_1\mathrm{d}k_2 .
\end{aligned}
\end{equation}
\subsubsection{Cutting off and splitting the collision operator on the collisional invariant region $\mathcal{S}(x)$}\label{Sec:Cutoff}
In this subsection, we follow the idea of \cite{CraciunBinh} to introduce a cut-off version for the collision operator $Q_c[f]$. The intuition behind this cut-off operator is explained below. We expect that as $t$ tends to infinity, the solution $f$ of \eqref{PhononEqC} converges to an equilibrium, which is a  function bounded from above and below by positive constants. Since the equilibrium is  bounded from above and below, it is not affected by the cut-off  operator. As a result, the solution $f$ is expected to be unchanged, under the effect of the cut-off operator, as $t$ goes to infinity.

Let $\varrho_N$ 
(for $0<N\le  \infty $)
be a function in  $C^1(\mathbb{R}_+)$ satisfying $\varrho_N[z]=1$ when $\frac{1}{N}\le z\le N$, $\varrho_N[z]=0$ when $0\le z\le \frac{1}{2N}$ and $z\ge 2N$,  and $0\le  \varrho_N[z]\le 1$ when $\frac{1}{2N}\le z\le \frac{1}{N}$ and $N\le z\le 2N$.
For  $f\in C^1(\mathcal{S}(x))$  and $0<N\le  \infty $,  define the cut-off function
\begin{equation}\label{ChiN}
\chi_N[f]  = \varrho_N[f]\varrho_N[|\nabla f|].
\end{equation}Note that $\chi_\infty[f]  = 1 $ for all $f\in C^1(\mathcal{S}(x))$.

We set the cut-off collision operator on the  collisional invariant region $\mathcal{S}(x)$ for $f$ and for $g$ defined in \eqref{gC}
\begin{equation}\label{CollisionCCutoff}
\begin{aligned}
& Q_c^N[f](k)= \\
\ & = \ \int_{\mathcal{S}(x)\times \mathcal{S}(x)}[\omega\omega_1\omega_2]^{-1}\chi_N^{*}\delta(k-k_1-k_2)\delta(\omega-\omega_1-\omega_2)[f_1f_2-ff_1-ff_2]\mathrm{d}k_1\mathrm{d}k_2\\
&\indent  - \ 2\int_{\mathcal{S}(x)\times \mathcal{S}(x)}[\omega\omega_1\omega_2]^{-1}\chi_N^{*}\delta(k_1-k-k_2)\delta(\omega_1-\omega-\omega_2)[f_2f-ff_1-f_1f_2]\mathrm{d}k_1\mathrm{d}k_2\\
\ & = \ \int_{\mathcal{S}(x)\times \mathcal{S}(x)}[\omega\omega_1\omega_2]^{-1}\chi_N^{*}[gg_1g_2]^{-1}\delta(k-k_1-k_2)\delta(\omega-\omega_1-\omega_2)[g-g_1-g_2]\mathrm{d}k_1\mathrm{d}k_2\\
&\indent  - \ 2\int_{\mathcal{S}(x)\times \mathcal{S}(x)}[\omega\omega_1\omega_2]^{-1}\chi_N^{*}[gg_1g_2]^{-1}\delta(k_1-k-k_2)\delta(\omega_1-\omega-\omega_2)[g_1-g_2-g]\mathrm{d}k_1\mathrm{d}k_2,
\end{aligned}
\end{equation} 
in which 
\begin{equation}
\label{ChiNStar}
\chi_N^{*}=\chi_N[f]\chi_N[f_1]\chi_N[f_2]=\chi_N[1/g]\chi_N[1/g_1]\chi_N[1/g_2].
\end{equation}
When $N=\infty$, we have that
\begin{equation}\label{CollisionCCutoffbis}
\begin{aligned}
Q_c^N[f](k)\ & = Q_c^\infty[f](k)\\ & = \ \int_{\mathcal{S}(x)\times \mathcal{S}(x)}[\omega\omega_1\omega_2]^{-1}\delta(k-k_1-k_2)\delta(\omega-\omega_1-\omega_2)[f_1f_2-ff_1-ff_2]\mathrm{d}k_1\mathrm{d}k_2\\
&\indent  - \ 2\int_{\mathcal{S}(x)\times \mathcal{S}(x)}[\omega\omega_1\omega_2]^{-1}\delta(k_1-k-k_2)\delta(\omega_1-\omega-\omega_2)[f_2f-ff_1-f_1f_2]\mathrm{d}k_1\mathrm{d}k_2\\
\ & = \ \int_{\mathcal{S}(x)\times \mathcal{S}(x)}[\omega\omega_1\omega_2]^{-1}[gg_1g_2]^{-1}\delta(k-k_1-k_2)\delta(\omega-\omega_1-\omega_2)[g-g_1-g_2]\mathrm{d}k_1\mathrm{d}k_2\\
&\indent  - \ 2\int_{\mathcal{S}(x)\times \mathcal{S}(x)}[\omega\omega_1\omega_2]^{-1}[gg_1g_2]^{-1}\delta(k_1-k-k_2)\delta(\omega_1-\omega-\omega_2)[g_1-g_2-g]\mathrm{d}k_1\mathrm{d}k_2.
\end{aligned}
\end{equation}

We also define the splitting collision operators on $\mathcal{S}(x)$, in which the kernel $[gg_1g_2]^{-1}$ is removed
\begin{equation}\label{GainCutC}
\begin{aligned}
\mathbb{Q}_c^{N,-}[g](k)
 \ =  & \ \int_{\mathcal{S}(x)\times \mathcal{S}(x)}\chi_N^{*}[\omega\omega_1\omega_2]^{-1}\delta(k-k_1-k_2)\delta(\omega-\omega_1-\omega_2){[g_1+g_2]}\mathrm{d}k_1\mathrm{d}k_2\\
& \ +2\int_{\mathcal{S}(x)\times \mathcal{S}(x)}\chi_N^{*}[\omega\omega_1\omega_2]^{-1}\delta(k_1-k-k_2)\delta(\omega_1-\omega-\omega_2){g_1}\mathrm{d}k_1\mathrm{d}k_2\\
& \ -2\int_{\mathcal{S}(x)\times \mathcal{S}(x)}\chi_N^{*}[\omega\omega_1\omega_2]^{-1}\delta(k_1-k-k_2)\delta(\omega_1-\omega-\omega_2){g_2}\mathrm{d}k_1\mathrm{d}k_2,
\end{aligned}
\end{equation}

\begin{equation}\label{LossCutC}
\begin{aligned}
 \ \mathbb{Q}_c^{N,+}[g](k) \ = & \ g\mathbb{L}_c^N(k)\\
= & \ g\int_{\mathcal{S}(x)\times \mathcal{S}(x)}\chi_N^{*}[\omega\omega_1\omega_2]^{-1}\delta(k-k_1-k_2)\delta(\omega-\omega_1-\omega_2)\mathrm{d}k_1\mathrm{d}k_2\\
& \ +2g\int_{\mathcal{S}(x)\times \mathcal{S}(x)}\chi_N^{*}[\omega\omega_1\omega_2]^{-1}\delta(k_1-k-k_2)\delta(\omega_1-\omega-\omega_2)\mathrm{d}k_1\mathrm{d}k_2,
\end{aligned}
\end{equation}
and
\begin{equation}
\label{DecompositionGainLossCutC}
\mathbb{Q}_c^N[g] \ = \  \mathbb{Q}_c^{N,+}[g] \ - \  \mathbb{Q}_0^{N,-}[g].
\end{equation}
Due to the symmetry of $k_1$ and $k_2$, $\mathbb{Q}_c^{N,-}[g](k)$ can be rewritten as
\begin{equation}\label{GainCutC2}
\begin{aligned}
&\mathbb{Q}_c^{N,-}[g](k) \ = \ \mathbb{Q}_c^{N,-,1}[g](k)+\mathbb{Q}_c^{N,-,2}[g](k)+\mathbb{Q}_c^{N,-,3}[g](k) \ :=  \\
 \ =  & \ 2\int_{\mathcal{S}(x)\times \mathcal{S}(x)}\chi_N^{*}[\omega\omega_1\omega_2]^{-1}\delta(k-k_1-k_2)\delta(\omega-\omega_1-\omega_2){g_1}\mathrm{d}k_1\mathrm{d}k_2\\
& \ +2\int_{\mathcal{S}(x)\times \mathcal{S}(x)}\chi_N^{*}[\omega\omega_1\omega_2]^{-1}\delta(k_1-k-k_2)\delta(\omega_1-\omega-\omega_2){g_1}\mathrm{d}k_1\mathrm{d}k_2\\
& \ -2\int_{\mathcal{S}(x)\times \mathcal{S}(x)}\chi_N^{*}[\omega\omega_1\omega_2]^{-1}\delta(k_1-k-k_2)\delta(\omega_1-\omega-\omega_2){g_2}\mathrm{d}k_1\mathrm{d}k_2.
\end{aligned}
\end{equation}
Note that in all of the above definitions, the cut-off parameter $N$ takes values  in the interval $(0,\infty].$
We then have the following lemma.
\begin{lemma}\label{Lemma:EquilibriumCutoff}
Given a collisional invariant region $\mathcal{S}(x)$, a function $\mathcal{F}^c(k)\in C(\mathcal{S}(x))$ is said to be a local equilibrium of  ${Q}_c^N$ on $\mathcal{S}(x)$ if and only if ${Q}_c^N[\mathcal{F}^c](k)=0$ and $\mathcal{F}^c(k)>0$ for all $k\in \mathcal{S}(x)$.

Under the local energy  constraint  
\begin{equation}\label{Lemma:EquilibriumCutoff:1}
\begin{aligned}
\int_{\mathcal{S}(x)}\mathcal{F}^c(k)\omega(k)\mathrm{d}k \ = & \ E_x 
\end{aligned}
\end{equation}
where $E_x$ is a given positive constant. Suppose that $ E_x \in \mathbb{R}_+$   and  
\begin{equation}\label{Lemma:EquilibriumCutoff:2}
\begin{aligned}
\int_{\mathcal{S}(x)}\frac{1}{a_x   }\mathrm{d}k \ = & \ E_x, 
\end{aligned}
\end{equation}
with $a_x\in\mathbb{R}_+$; the local equilibrium on $\mathcal{S}(x)$ can be uniquely determined, when $N$ is sufficiently large, as
\begin{equation}\label{Lemma:EquilibriumCutoff:3}
\mathcal{F}^c(k) \ = \ \frac{1}{a_x\omega(k)  }.
\end{equation}
Similarly, a function $\mathcal{E}^c(k)$  is said to be a local equilibrium of  $\mathbb{Q}_c^N$ on $\mathcal{S}(x)$ if and only if $\mathbb{Q}_c^N[\mathcal{F}^c](k)=0$ and $$\mathcal{E}^c(k)=a_x\omega(k).$$
\end{lemma}

\begin{proof}
The proof follows from the same lines of arguments used in the proof of Proposition \ref{Propo:ClassicalEqui}.
\end{proof}
\subsection{The long time dynamics of solutions to the 3-wave kinetic equation on non-collision and collisional invariant regions}

\subsubsection{An estimate on the distance between $f$ and $\mathcal{F}^c$}\label{Sec:Distance}
This section is devoted to the estimate of the difference between a function $f$ and a local equilibrium $\mathcal{F}^c$, defined on the same collisional invariant region. The two functions $f$ and $\mathcal{F}^c$ are supposed to have the same energy. 
\begin{proposition}\label{Propo:EntropyC}
Let $\mathcal{S}(x)$ be a collisonal invariant region and $f$ be a positive function  such that $f\in L^1(\mathcal{S}(x))$. Let 
\begin{equation}
\label{Propo:Entropy:EquilibriumC}
\mathcal{F}^c(k)  \ = \ \frac{1}{a_x\omega(k) } \ = : \ \frac{1}{\mathcal{E}^c(k)} ,
\end{equation}
where  $a_x\in\mathbb{R}$ satisfying $\mathcal{F}^c(k) >0$ for all $k\in\mathcal{S}(x)$. 

In addition, we assume
\begin{equation}
\label{Propo:Entropy:ConservationC}
\int_{\mathcal{S}(x)}f(k)\omega(k)\mathrm{d}k \ = \  \int_{\mathcal{S}(x)}\mathcal{F}(k)\omega(k)\mathrm{d}k.
\end{equation}

We also define $g$ using \eqref{gC}.

 Then, the following inequalities always hold true for $0\le N\le\infty$
 \begin{equation}\label{Propo:EntropyC:0}
\begin{aligned}
& \int_{\mathcal{S}(x)}\sqrt{f\left|\mathbb{Q}_c^{N,+}[g]-\mathbb{Q}_c^{N,-}[g]\right|}\mathrm{d}k \ \lesssim  \ \left[ \int_{\mathcal{S}(x)}f\mathrm{d}k\right]^\frac12\times\\
& \times\left[\int_{\mathcal{S}(x)\times \mathcal{S}(x)\times \mathcal{S}(x)}[\omega\omega_1\omega_2]^{-1}\chi_N^{*}\delta(k-k_1-k_2)\delta(\omega-\omega_1-\omega_2)|g-g_1-g_2|^2\mathrm{d}k\mathrm{d}k_1\mathrm{d}k_2\right]^\frac14,
\end{aligned}
\end{equation}
and
\begin{equation}\label{Propo:EntropyC:1}
\begin{aligned}
& \ \left\|\sqrt{\mathbb{L}_c^N\mathcal{E}^c|f-\mathcal{F}^c|} \right\|_{L^1(\mathcal{S}(x))} \ \lesssim   \ \left[ \int_{\mathcal{S}(x)}f\mathrm{d}k\right]^\frac12\left\{\|g-\mathcal{E}^c\|_{L^1(\mathcal{S}(x))}^\frac12\ + \ \right.\\
&\int_{\mathcal{S}(x)\times \mathcal{S}(x)\times \mathcal{S}(x)}[\omega\omega_1\omega_2]^{-1} \left.\left. \chi_N^{*}\delta(k-k_1-k_2)\delta(\omega-\omega_1-\omega_2)|g-g_1-g_2|^2\mathrm{d}k\mathrm{d}k_1\mathrm{d}k_2\right]^\frac14\right\}
\end{aligned}
\end{equation} 
in which the constants on the right hand sides do not depend on $f$.
\end{proposition}

\begin{proof}
Considering the difference between $f$ and $\mathcal{F}^c$ on $\mathcal{S}(x)$, we find
$$|f-\mathcal{F}^c| \ = \ \left|\frac{1}{g}-\frac{1}{\mathcal{E}^c}\right| \ = \ \frac{|g-\mathcal{E}^c|}{g\mathcal{E}^c},$$
which then implies
$$\mathcal{E}^c|f-\mathcal{F}^c| \ = \ f|g-\mathcal{E}^c|.$$

Multiplying both sides with $\mathbb{L}_c^N$ and taking the square yields
$$\sqrt{\mathbb{L}_c^N\mathcal{E}^c|f-\mathcal{F}^c|} \ = \ \sqrt{\mathbb{L}_c^Nf|g-\mathcal{E}^c|},$$
which, by the fact that $\mathbb{L}_c^Ng= \mathbb{Q}_c^{N,+}[g]$ and  $\mathbb{L}_c^N\mathcal{E}^c= \mathbb{Q}_c^{N,+}[\mathcal{E}^c]$, implies
$$\sqrt{\mathbb{L}_c^N\mathcal{E}^c|f-\mathcal{F}^c|} \ = \ \sqrt{f\left|\mathbb{Q}_c^{N,+}[g]-\mathbb{Q}_c^{N,+}[\mathcal{E}^c]\right|}.$$
Applying the triangle inequality to the right hand side gives
$$\begin{aligned}
\sqrt{\mathbb{L}_c^N\mathcal{E}^c|f-\mathcal{F}^c|} \ \lesssim & \ \sqrt{f\left|\mathbb{Q}_c^{N,+}[g]-\mathbb{Q}_c^{N,-}[g]\right|} \ + \ \sqrt{f\left|\mathbb{Q}_c^{N,-}[g]-\mathbb{Q}_c^{N,-}[\mathcal{E}^c]\right|}\\
& \ + \ \sqrt{f\left|\mathbb{Q}_c^{N,+}[\mathcal{E}^c]-\mathbb{Q}_c^{N,-}[\mathcal{E}^c]\right|}.
\end{aligned}$$
By Lemma \ref{Lemma:EquilibriumCutoff}, the last term on the right hand side of the above inequality vanishes, yielding
\begin{equation}
\label{Propo:EntropyC:E1}
\begin{aligned}
\sqrt{\mathbb{L}_c^N\mathcal{E}^c|f-\mathcal{F}^c|} \ \lesssim & \ \sqrt{f\left|\mathbb{Q}_c^{N,+}[g]-\mathbb{Q}_c^{N,-}[g]\right|} \ + \ \sqrt{f\left|\mathbb{Q}_c^{N,-}[g]-\mathbb{Q}_c^{N,-}[\mathcal{E}^c]\right|}.\end{aligned}\end{equation}
Integrating the first term on the right hand side and using H\"older's inequality leads to
\begin{equation}\label{Propo:EntropyC:E2}
\begin{aligned}
\left(\int_{\mathcal{S}(x)}\sqrt{f\left|\mathbb{Q}_c^{N,+}[g]-\mathbb{Q}_c^{N,-}[g]\right|}\mathrm{d}k\right)^2 \ \le \ \left( \int_{\mathcal{S}(x)}f\mathrm{d}k\right)\left(\int_{\mathcal{S}(x)}\left|\mathbb{Q}_c^{N,+}[g]-\mathbb{Q}_c^{N,-}[g]\right|\mathrm{d}k\right).
\end{aligned}
\end{equation}
Observe that
\begin{equation*}
\begin{aligned}
& \ \left|\mathbb{Q}_c^{N,+}[g]-\mathbb{Q}_c^{N,-}[g]\right|\  \le \ \\
\  \le \ & \int_{\mathcal{S}(x)\times \mathcal{S}(x)}[\omega\omega_1\omega_2]^{-1}\chi_N^{*}\delta(k-k_1-k_2)\delta(\omega-\omega_1-\omega_2)|g-g_1-g_2|\mathrm{d}k_1\mathrm{d}k_2\\
& + \ 2\int_{\mathcal{S}(x)\times \mathcal{S}(x)}[\omega\omega_1\omega_2]^{-1}\chi_N^{*}\delta(k_1-k-k_2)\delta(\omega_1-\omega-\omega_2)|g_1-g_2-g|\mathrm{d}k_1\mathrm{d}k_2,
\end{aligned}
\end{equation*} 
which, after integrating in $k$ and taking into account the symmetry of $k,k_1,k_2$, yields
\begin{equation*}
\begin{aligned}
& \ \int_{\mathcal{S}(x)} \left|\mathbb{Q}_c^{N,+}[g]-\mathbb{Q}_c^{N,-}[g]\right|\mathrm{d}k\  \le \ \\
\  \le \ & 3\int_{\mathcal{S}(x)\times \mathcal{S}(x)\times \mathcal{S}(x)}[\omega\omega_1\omega_2]^{-1}\chi_N^{*}\delta(k-k_1-k_2)\delta(\omega-\omega_1-\omega_2)|g-g_1-g_2|\mathrm{d}k\mathrm{d}k_1\mathrm{d}k_2.
\end{aligned}
\end{equation*} 
Applying H\"older's inequality again to the right hand side implies
 \begin{equation}\label{Propo:EntropyC:E3a}
\begin{aligned}
& \ \int_{\mathcal{S}(x)} \left|\mathbb{Q}_c^{N,+}[g]-\mathbb{Q}_c^{N,-}[g]\right|\mathrm{d}k\  \le \ \\
\  \le \ & 3\left[\int_{\mathcal{S}(x)\times \mathcal{S}(x)\times \mathcal{S}(x)}[\omega\omega_1\omega_2]^{-1}\chi_N^{*}\delta(k-k_1-k_2)\delta(\omega-\omega_1-\omega_2)\mathrm{d}k\mathrm{d}k_1\mathrm{d}k_2\right]^\frac12\times\\
& \times \left[\int_{\mathcal{S}(x)\times \mathcal{S}(x)\times \mathcal{S}(x)}[\omega\omega_1\omega_2]^{-1}\chi_N^{*}\delta(k-k_1-k_2)\delta(\omega-\omega_1-\omega_2)|g-g_1-g_2|^2\mathrm{d}k\mathrm{d}k_1\mathrm{d}k_2\right]^\frac12.
\end{aligned}
\end{equation} 
Using the fact that $\chi^*_N\le 1$, Corollary \ref{Coll:Edges} and Proposition \ref{Lemma:RestrictionSx} to bound the integral containing only $[\omega\omega_1\omega_2]^{-1}\chi_N^{*}\delta(k-k_1-k_2)\delta(\omega-\omega_1-\omega_2)$, we derive from the above inequality the following estimate
\begin{equation}\label{Propo:EntropyC:E3}
\begin{aligned}
& \ \int_{\mathcal{S}(x)} \left|\mathbb{Q}_c^{N,+}[g]-\mathbb{Q}_c^{N,-}[g]\right|\mathrm{d}k\  \le \ \\
\  \le \ & 3\left[\int_{\mathcal{S}(x)\times \mathcal{S}(x)\times \mathcal{S}(x)}[\omega\omega_1\omega_2]^{-1}\chi_N^{*}\delta(k-k_1-k_2)\delta(\omega-\omega_1-\omega_2)\mathrm{d}k\mathrm{d}k_1\mathrm{d}k_2\right]^\frac12\times\\
& \times \left[\int_{\mathcal{S}(x)\times \mathcal{S}(x)\times \mathcal{S}(x)}[\omega\omega_1\omega_2]^{-1}\chi_N^{*}\delta(k-k_1-k_2)\delta(\omega-\omega_1-\omega_2)|g-g_1-g_2|^2\mathrm{d}k\mathrm{d}k_1\mathrm{d}k_2\right]^\frac12\\
\  \lesssim \ & \left[\int_{\mathcal{S}(x)\times \mathcal{S}(x)\times \mathcal{S}(x)}[\omega\omega_1\omega_2]^{-1}\chi_N^{*}\delta(k-k_1-k_2)\delta(\omega-\omega_1-\omega_2)|g-g_1-g_2|^2\mathrm{d}k\mathrm{d}k_1\mathrm{d}k_2\right]^\frac12.
\end{aligned}
\end{equation} 
Putting \eqref{Propo:EntropyC:E2} and \eqref{Propo:EntropyC:E3}  together, we obtain
\begin{equation}\label{Propo:EntropyC:E4}
\begin{aligned}
& \int_{\mathcal{S}(x)}\sqrt{f\left|\mathbb{Q}_c^{N,+}[g]-\mathbb{Q}_c^{N,-}[g]\right|}\mathrm{d}k \ \lesssim  \ \left[ \int_{\mathcal{S}(x)}f\mathrm{d}k\right]^\frac12\times\\
& \times\left[\int_{\mathcal{S}(x)\times \mathcal{S}(x)\times \mathcal{S}(x)}[\omega\omega_1\omega_2]^{-1}\chi_N^{*}\delta(k-k_1-k_2)\delta(\omega-\omega_1-\omega_2)|g-g_1-g_2|^2\mathrm{d}k\mathrm{d}k_1\mathrm{d}k_2\right]^\frac14.
\end{aligned}
\end{equation}
Integrating the second term on the right hand side of \eqref{Propo:EntropyC:E1} and using H\"older's inequality 
\begin{equation}\label{Propo:EntropyC:E5}
\begin{aligned}
\left(\int_{\mathcal{S}(x)}\sqrt{f\left|\mathbb{Q}_c^{N,-}[g]-\mathbb{Q}_c^{N,-}[\mathcal{E}^c]\right|}\mathrm{d}k\right)^2 \ \le \ \left( \int_{\mathcal{S}(x)}f\mathrm{d}k\right)\left(\int_{\mathcal{S}(x)}\left|\mathbb{Q}_c^{N,-}[g]-\mathbb{Q}_c^{N,-}[\mathcal{E}^c]\right|\mathrm{d}k\right).
\end{aligned}
\end{equation}
It is straightforward that 
\begin{equation*}
\begin{aligned}
& \ \left|\mathbb{Q}_c^{N,-}[g]-\mathbb{Q}_c^{N,-}[\mathcal{E}^c]\right|\  \le \ \\
\  \le \ & \int_{\mathcal{S}(x)\times \mathcal{S}(x)}[\omega\omega_1\omega_2]^{-1}\chi_N^{*}\delta(k-k_1-k_2)\delta(\omega-\omega_1-\omega_2)[|g_1-\mathcal{E}_1^c|+|g_2-\mathcal{E}_2^c|]\mathrm{d}k_1\mathrm{d}k_2\\
& + \ 2\int_{\mathcal{S}(x)\times \mathcal{S}(x)}[\omega\omega_1\omega_2]^{-1}\chi_N^{*}\delta(k_1-k-k_2)\delta(\omega_1-\omega-\omega_2)|g_1-\mathcal{E}_1^c|\mathrm{d}k_1\mathrm{d}k_2\\
& + \ 2\int_{\mathcal{S}(x)\times \mathcal{S}(x)}[\omega\omega_1\omega_2]^{-1}\chi_N^{*}\delta(k_1-k-k_2)\delta(\omega_1-\omega-\omega_2)|g_2-\mathcal{E}_2^c|\mathrm{d}k_1\mathrm{d}k_2.
\end{aligned}
\end{equation*} 
Integrating in $k$, we immediately find
\begin{equation*}
\begin{aligned}
& \ \int_{\mathcal{S}(x)}\left|\mathbb{Q}_c^{N,-}[g]-\mathbb{Q}_c^{N,-}[\mathcal{E}^c]\right|\mathrm{d}k\  \le \ \\
\  \le \ & \int_{\mathcal{S}(x)\times \mathcal{S}(x)\times \mathcal{S}(x)}[\omega\omega_1\omega_2]^{-1}\chi_N^{*}\delta(k-k_1-k_2)\delta(\omega-\omega_1-\omega_2)[|g_1-\mathcal{E}_1^c|+|g_2-\mathcal{E}_2^c|]\mathrm{d}k\mathrm{d}k_1\mathrm{d}k_2\\
& + \ 2\int_{\mathcal{S}(x)\times \mathcal{S}(x)}[\omega\omega_1\omega_2]^{-1}\chi_N^{*}\delta(k_1-k-k_2)\delta(\omega_1-\omega-\omega_2)|g_1-\mathcal{E}_1^c|\mathrm{d}k\mathrm{d}k_1\mathrm{d}k_2\\
& + \ 2\int_{\mathcal{S}(x)\times \mathcal{S}(x)\times \mathcal{S}(x)}[\omega\omega_1\omega_2]^{-1}\chi_N^{*}\delta(k_1-k-k_2)\delta(\omega_1-\omega-\omega_2)|g_2-\mathcal{E}_2^c|\mathrm{d}k\mathrm{d}k_1\mathrm{d}k_2,
\end{aligned}
\end{equation*} 
which, by  the symmetry between $k_1$ and $k_2$ and the fact that $\chi^*_N\le 1$, implies
\begin{equation*}
\begin{aligned}
& \ \int_{\mathcal{S}(x)}\left|\mathbb{Q}_c^{N,-}[g]-\mathbb{Q}_c^{N,-}[\mathcal{E}^c]\right|\mathrm{d}k\  \le \ \\
\  \le \ & 2\int_{\mathcal{S}(x)\times \mathcal{S}(x)\times \mathcal{S}(x)}[\omega\omega_1\omega_2]^{-1}\delta(k-k_1-k_2)\delta(\omega-\omega_1-\omega_2)|g_1-\mathcal{E}_1^c|\mathrm{d}k\mathrm{d}k_1\mathrm{d}k_2\\
& + \ 2\int_{\mathcal{S}(x)\times \mathcal{S}(x)}[\omega\omega_1\omega_2]^{-1}\delta(k_1-k-k_2)\delta(\omega_1-\omega-\omega_2)|g_1-\mathcal{E}_1^c|\mathrm{d}k\mathrm{d}k_1\mathrm{d}k_2\\
& + \ 2\int_{\mathcal{S}(x)\times \mathcal{S}(x)\times \mathcal{S}(x)}[\omega\omega_1\omega_2]^{-1}\delta(k_1-k-k_2)\delta(\omega_1-\omega-\omega_2)|g_2-\mathcal{E}_2^c|\mathrm{d}k\mathrm{d}k_1\mathrm{d}k_2.
\end{aligned}
\end{equation*} 
Now, we can also combine the last and the first terms on the right hand side using the change of variables between $k,k_1,k_2$ to get
\begin{equation}\label{Propo:EntropyC:E6}
\begin{aligned}
& \ \int_{\mathcal{S}(x)}\left|\mathbb{Q}_c^{N,-}[g]-\mathbb{Q}_c^{N,-}[\mathcal{E}^c]\right|\mathrm{d}k\  \le \ \\
\  \le \ & 4\int_{\mathcal{S}(x)\times \mathcal{S}(x)\times \mathcal{S}(x)}[\omega\omega_1\omega_2]^{-1}\delta(k-k_1-k_2)\delta(\omega-\omega_1-\omega_2)|g_1-\mathcal{E}_1^c|\mathrm{d}k\mathrm{d}k_1\mathrm{d}k_2\\
& + \ 2\int_{\mathcal{S}(x)\times \mathcal{S}(x)}[\omega\omega_1\omega_2]^{-1}\delta(k_1-k-k_2)\delta(\omega_1-\omega-\omega_2)|g_1-\mathcal{E}_1^c|\mathrm{d}k\mathrm{d}k_1\mathrm{d}k_2.
\end{aligned}
\end{equation} 
Let us  estimate  each term on the right hand side of \eqref{Propo:EntropyC:E6}. 

Taking the integration in $k_2$ of the first term yields
\begin{equation*}
\begin{aligned}
& 4\int_{\mathcal{S}(x)\times \mathcal{S}(x)\times \mathcal{S}(x)}[\omega\omega_1\omega_2]^{-1}\delta(k-k_1-k_2)\delta(\omega-\omega_1-\omega_2)|g_1-\mathcal{E}_1^c|\mathrm{d}k\mathrm{d}k_1\mathrm{d}k_2\\
= \ & 4\int_{\mathcal{S}(x)\times \mathcal{S}(x)}[\omega(k)\omega(k_1)\omega(k-k_1)]^{-1}\delta(\omega(k)-\omega(k_1)-\omega(k-k_1))|g_1-\mathcal{E}_1^c|\mathrm{d}k\mathrm{d}k_1.
\end{aligned}
\end{equation*} 
Observing that $\omega(k)\ge\omega_0>0$ for all $k\in\mathbb{T}^3$, we find
\begin{equation*}
\begin{aligned}
& 4\int_{\mathcal{S}(x)\times \mathcal{S}(x)\times \mathcal{S}(x)}[\omega\omega_1\omega_2]^{-1}\delta(k-k_1-k_2)\delta(\omega-\omega_1-\omega_2)|g_1-\mathcal{E}_1^c|\mathrm{d}k\mathrm{d}k_1\mathrm{d}k_2\\
\lesssim \ & \int_{\mathcal{S}(x)\times \mathcal{S}(x)}\delta(\omega(k)-\omega(k_1)-\omega(k-k_1))|g_1-\mathcal{E}_1^c|\mathrm{d}k\mathrm{d}k_1,
\end{aligned}
\end{equation*} 
which, after integrating with respect to $k_1$, leads to
\begin{equation*}
\begin{aligned}
& 4\int_{\mathcal{S}(x)\times \mathcal{S}(x)\times \mathcal{S}(x)}[\omega\omega_1\omega_2]^{-1}\delta(k-k_1-k_2)\delta(\omega-\omega_1-\omega_2)|g_1-\mathcal{E}_1^c|\mathrm{d}k\mathrm{d}k_1\mathrm{d}k_2\\
\lesssim \ & \int_{\mathcal{S}(x)}\left[\int_{\mathcal{S}(x)}\delta(\omega(k)-\omega(k_1)-\omega(k-k_1))\mathrm{d}k\right]|g_1-\mathcal{E}_1^c|\mathrm{d}k_1.
\end{aligned}
\end{equation*} 
Note that the integration with respect to $k$ is uniformly bounded in $k_1\in\mathbb{T}^3$ by Corollary \ref{Coll:Edges} and Proposition \ref{Lemma:RestrictionSx}, we then get
\begin{equation}\label{Propo:EntropyC:E7}
\begin{aligned}
& 4\int_{\mathcal{S}(x)\times \mathcal{S}(x)\times \mathcal{S}(x)}[\omega\omega_1\omega_2]^{-1}\delta(k-k_1-k_2)\delta(\omega-\omega_1-\omega_2)|g_1-\mathcal{E}_1^c|\mathrm{d}k\mathrm{d}k_1\mathrm{d}k_2\\
\lesssim \ & \int_{\mathcal{S}(x)}|g_1-\mathcal{E}_1^c|\mathrm{d}k_1 \ = \ \|g-\mathcal{E}^c\|_{L^1(\mathcal{S}(x))}.
\end{aligned}
\end{equation} 

The second term on the right hand side of \eqref{Propo:EntropyC:E6} can also be estimated in the same way. Taking the integration in $k_2$ of the second term yields
\begin{equation*}
\begin{aligned}
& 2\int_{\mathcal{S}(x)\times \mathcal{S}(x)\times \mathcal{S}(x)}[\omega\omega_1\omega_2]^{-1}\delta(k_1-k-k_2)\delta(\omega_1-\omega-\omega_2)|g_1-\mathcal{E}_1^c|\mathrm{d}k\mathrm{d}k_1\mathrm{d}k_2\\
= \ & 2\int_{\mathcal{S}(x)\times \mathcal{S}(x)}[\omega(k)\omega(k_1)\omega(k-k_1)]^{-1}\delta(\omega(k_1)-\omega(k)-\omega(k_1-k))|g_1-\mathcal{E}_1^c|\mathrm{d}k\mathrm{d}k_1,
\end{aligned}
\end{equation*} 
which, similarly as above, can be bounded as
\begin{equation*}
\begin{aligned}
& 2\int_{\mathcal{S}(x)\times \mathcal{S}(x)\times \mathcal{S}(x)}[\omega\omega_1\omega_2]^{-1}\delta(k_1-k-k_2)\delta(\omega_1-\omega-\omega_2)|g_1-\mathcal{E}_1^c|\mathrm{d}k\mathrm{d}k_1\mathrm{d}k_2\\
\lesssim \ & \int_{\mathcal{S}(x)}\left[\int_{\mathcal{S}(x)}\delta(\omega(k_1)-\omega(k)-\omega(k_1-k))\mathrm{d}k\right]|g_1-\mathcal{E}_1^c|\mathrm{d}k_1.
\end{aligned}
\end{equation*} 
Again, the integration with respect to $k$ is bounded, we therefore have
\begin{equation}\label{Propo:EntropyC:E8}
\begin{aligned}
& 4\int_{\mathcal{S}(x)\times \mathcal{S}(x)\times \mathcal{S}(x)}[\omega\omega_1\omega_2]^{-1}\delta(k_1-k-k_2)\delta(\omega_1-\omega-\omega_2)|g_1-\mathcal{E}_1^c|\mathrm{d}k\mathrm{d}k_1\mathrm{d}k_2\\
\lesssim \ & \int_{\mathcal{S}(x)}|g_1-\mathcal{E}_1^c|\mathrm{d}k_1 \ = \ \|g-\mathcal{E}^c\|_{L^1(\mathcal{S}(x))}.
\end{aligned}
\end{equation} 
Now, combining \eqref{Propo:EntropyC:E5},\eqref{Propo:EntropyC:E6}, \eqref{Propo:EntropyC:E7}, \eqref{Propo:EntropyC:E8} leads to
\begin{equation}\label{Propo:EntropyC:E9}
\begin{aligned}
& \ \int_{\mathcal{S}(x)}\sqrt{f\left|\mathbb{Q}_c^{N,-}[g]-\mathbb{Q}_c^{N,-}[\mathcal{E}^c]\right|}\mathrm{d}k\  \lesssim \ \\
\lesssim \ &  \left[ \int_{\mathcal{S}(x)}f\mathrm{d}k\right]^\frac12\left[\int_{\mathcal{S}(x)}|g_1-\mathcal{E}_1^c|\mathrm{d}k_1\right]^\frac12 \ = \  \left[ \int_{\mathcal{S}(x)}f\mathrm{d}k\right]^\frac12\|g-\mathcal{E}^c\|_{L^1(\mathcal{S}(x))}^\frac12.
\end{aligned}
\end{equation} 

Putting together the three estimates \eqref{Propo:EntropyC:E1},\eqref{Propo:EntropyC:E4} and \eqref{Propo:EntropyC:E9} yields
\begin{equation}\label{Propo:EntropyC:E10}
\begin{aligned}
& \ \left\|\sqrt{\mathbb{L}_c^N\mathcal{E}^c|f-\mathcal{F}^c|} \right\|_{L^1(\mathcal{S}(x))} \ \lesssim   \ \left[ \int_{\mathcal{S}(x)}f\mathrm{d}k\right]^\frac12\|g-\mathcal{E}^c\|_{L^1(\mathcal{S}(x))}^\frac12\ + \ \left[ \int_{\mathcal{S}(x)}f\mathrm{d}k\right]^\frac12\times \\
& \times\left[\int_{\mathcal{S}(x)\times \mathcal{S}(x)\times \mathcal{S}(x)}[\omega\omega_1\omega_2]^{-1}\chi_N^{*}\delta(k-k_1-k_2)\delta(\omega-\omega_1-\omega_2)|g-g_1-g_2|^2\mathrm{d}k\mathrm{d}k_1\mathrm{d}k_2\right]^\frac14
\end{aligned}
\end{equation} 
\end{proof}

\subsubsection{A lower bound on the solution of the equation with the cut-off collision operator on the collisional invariant region $\mathcal{S}(x)$}\label{Sec:LowerBound}
The following Proposition provides a uniform lower bound to  classical solutions of the wave kinetic equation on $\mathcal{S}(x)$, under the effect of the cut-off operator $\chi_N$.
\begin{proposition}\label{Propo:LowerBound}
Suppose that the initial condition $f_0$ of \eqref{PhononEqC}  is bounded from below by a strictly positive constant ${f_0^*}$, and $f_0\in C(\mathcal{S}(x))$. Let $f$ be a classical solution in $C^0([0,\infty), C(\mathcal{S}(x)))$ $\cap$ $ C^1((0,\infty), C(\mathcal{S}(x)))$ to  \eqref{PhononEqC} .
There exists a strictly positive function ${f^*}(t)>0$, which is non-increasing in $t$, such that $f(t,k)>{f^*}(t)>0$ for all $k\in\mathcal{S}(x)$ and for all $t\ge0$. To be more precise, there exists a universal constant ${f}_*>0$ such that 
$$f(t,k)>{f^*}(t)=\frac{{f}_*}{\sup_{s\in[0,t]}\|f(s,\cdot)\|_{C(\mathcal{S}(x))}}.$$
\end{proposition}
\begin{proof}
Rearranging the equation, one finds
\begin{equation*}
\begin{aligned}
\partial_t f \ =  \ & \int_{\mathcal{S}(x)\times\mathcal{S}(x)}[\omega\omega_1\omega_2]^{-1}\delta(k-k_1-k_2)\delta(\omega-\omega_1-\omega_2)f_1f_2\mathrm{d}k_1\mathrm{d}k_2\\
& + \  2\int_{\mathcal{S}(x)\times\mathcal{S}(x)}[\omega\omega_1\omega_2]^{-1}\delta(k_1-k-k_2)\delta(\omega_1-\omega-\omega_2)[f_1f_2+ff_1]\mathrm{d}k_1\mathrm{d}k_2\\
& - \ f\left[\int_{\mathcal{S}(x)\times\mathcal{S}(x)}[\omega\omega_1\omega_2]^{-1}\delta(k-k_1-k_2)\delta(\omega-\omega_1-\omega_2)(f_1+f_2)\mathrm{d}k_1\mathrm{d}k_2\right.\\
& + \ \left.2\int_{\mathcal{S}(x)\times\mathcal{S}(x)}[\omega\omega_1\omega_2]^{-1}\delta(k_1-k-k_2)\delta(\omega_1-\omega-\omega_2)f_2\mathrm{d}k_1\mathrm{d}k_2\right].
\end{aligned}
\end{equation*}
Using the symmetry of $f_1$ and $f_2$ in the term containing $f_1+f_2$, we can turn this term into a new term, in which $f_1+f_2$ is replaced by $2f_1$ \begin{equation}\label{Propo:LowerBound:E1}
\begin{aligned}
\partial_t f \ =  \ & \int_{\mathcal{S}(x)\times\mathcal{S}(x)}[\omega\omega_1\omega_2]^{-1}\delta(k-k_1-k_2)\delta(\omega-\omega_1-\omega_2)f_1f_2\mathrm{d}k_1\mathrm{d}k_2\\
& + \  2\int_{\mathcal{S}(x)\times\mathcal{S}(x)}[\omega\omega_1\omega_2]^{-1}\delta(k_1-k-k_2)\delta(\omega_1-\omega-\omega_2)[f_1f_2+ff_1]\mathrm{d}k_1\mathrm{d}k_2\\
& - \ 2f\left[\int_{\mathcal{S}(x)\times\mathcal{S}(x)}[\omega\omega_1\omega_2]^{-1}\delta(k-k_1-k_2)\delta(\omega-\omega_1-\omega_2)f_1\mathrm{d}k_1\mathrm{d}k_2\right.\\
& + \ \left.\int_{\mathcal{S}(x)\times\mathcal{S}(x)}[\omega\omega_1\omega_2]^{-1}\delta(k_1-k-k_2)\delta(\omega_1-\omega-\omega_2)f_2\mathrm{d}k_1\mathrm{d}k_2\right].
\end{aligned}
\end{equation}
Now, let us consider the term with the minus sign
\begin{equation}\label{Propo:LowerBound:E1a}
\begin{aligned}
&  2f\left[\int_{\mathcal{S}(x)\times\mathcal{S}(x)}[\omega\omega_1\omega_2]^{-1}\delta(k-k_1-k_2)\delta(\omega-\omega_1-\omega_2)f_1\mathrm{d}k_1\mathrm{d}k_2\right.\\
& + \ \left.\int_{\mathcal{S}(x)\times\mathcal{S}(x)}[\omega\omega_1\omega_2]^{-1}\delta(k_1-k-k_2)\delta(\omega_1-\omega-\omega_2)f_2\mathrm{d}k_1\mathrm{d}k_2\right].
\end{aligned}
\end{equation}

We  define the function $\mathbb{B}: \mathbb{R}_+\to\mathbb{R}_+$
\begin{equation}
\mathbb{B}(t) \ = \ \sup_{s\in[0,t]}\|f(s,\cdot)\|_{C(\mathcal{S}(x))},
\end{equation}
which is an increasing function in $t$. Using the fact that $\omega\ge \omega_0>0$ and the function $ \mathbb{B}(t)$, we can  bound $\eqref{Propo:LowerBound:E1a}$  from above by
\begin{equation*}
\begin{aligned}
&  \frac{2\mathbb{B}(t) }{\omega_0^3}f\left[\int_{\mathcal{S}(x)\times\mathcal{S}(x)}\delta(k-k_1-k_2)\delta(\omega-\omega_1-\omega_2)\mathrm{d}k_1\mathrm{d}k_2\right.\\
& + \ \left.\int_{\mathcal{S}(x)\times\mathcal{S}(x)}\delta(k_1-k-k_2)\delta(\omega_1-\omega-\omega_2)\mathrm{d}k_1\mathrm{d}k_2\right].
\end{aligned}
\end{equation*}

Integrating in $k_2$ and using the definite of the two delta functions $\delta(k-k_1-k_2)$ and $\delta(k_1-k-k_2)$
\begin{equation*}
\begin{aligned}
&
\ \frac{2\mathbb{B}(t) }{\omega_0^3}f(k)\left[\int_{\mathcal{S}(x)}\delta(\omega(k)-\omega(k_1)-\omega(k-k_1))\mathrm{d}k_1\right.\\
& \left.+ \ \int_{\mathcal{S}(x)}\delta(\omega(k)-\omega(k_1)-\omega(k-k_1))\mathrm{d}k_1\right]\
 \le  \ \frac{2\mathbb{B}(t) }{\omega_0^3}\mathfrak{C}_1f(k) \ =: \ \mathcal{C}(t)f(k).
 \end{aligned}
\end{equation*}

We therefore obtain the following bound for $\partial_t f$
\begin{equation}\label{Propo:LowerBound:E2}
\begin{aligned}
\partial_t f \ \ge  \ & \int_{\mathcal{S}(x)\times\mathcal{S}(x)}[\omega\omega_1\omega_2]^{-1}\delta(k-k_1-k_2)\delta(\omega-\omega_1-\omega_2)f_1f_2\mathrm{d}k_1\mathrm{d}k_2\\
& + \  2\int_{\mathcal{S}(x)\times\mathcal{S}(x)}[\omega\omega_1\omega_2]^{-1}\delta(k_1-k-k_2)\delta(\omega_1-\omega-\omega_2)[f_1f_2+ff_1]\mathrm{d}k_1\mathrm{d}k_2\\
& - \ \mathcal{C}(t)f.
\end{aligned}
\end{equation}
Define the positive terms on the right hand side by $K[f]$, we then have the simplified equation
\begin{equation}\label{Propo:LowerBound:E3}
\begin{aligned}
\partial_t f \ \ge  \ & K[f] - \ \mathcal{C}(t)f,
\end{aligned}
\end{equation}
which, by Duhamel's formula and the mononicity in $t$ of $\mathcal{C}(t)$, gives
\begin{equation}\label{Propo:LowerBound:E4}
\begin{aligned}
f(t,k) \ \ge  \ & f_0(k)e^{-\mathcal{C}(T)t} \ + \ \int_0^tK[f](t-s,k)e^{-\mathcal{C}(T) (t-s)}\mathrm{d}s,
\end{aligned}
\end{equation}

Using the fact that $f_0(k)\ge {f_0^*}>0$, we deduce from \eqref{Propo:LowerBound:E4} the following estimate
\begin{equation}\label{Propo:LowerBound:E5}
\begin{aligned}
f(t,k) \ \ge  \ & {f_0^*}e^{-\mathcal{C}(T)t} \ + \ \int_0^tK[f](t-s,k)e^{-\mathcal{C}(T)(t-s)}\mathrm{d}s.
\end{aligned}
\end{equation}
We observe that the second term on the right hand side is always positive, since it contains only positive components. This implies 
\begin{equation}\label{Propo:LowerBound:E6}
\begin{aligned}
f(t,k) \ \ge  \ & {f_0^*}e^{-\mathcal{C}(T)t},
\end{aligned}
\end{equation}
for all $t\in[0,T]$.

Now, let us examine the operator $K[f]$ in details. Using the fact $\omega\le \omega_0 + 12$, we can bound $K[f]$ as 
\begin{equation*}
\begin{aligned}
K[f] \ \ge \ & [\omega_0 + 12]^{-3}\left[\int_{\mathcal{S}(x)\times\mathcal{S}(x)}\delta(k-k_1-k_2)\delta(\omega-\omega_1-\omega_2)f_1f_2\mathrm{d}k_1\mathrm{d}k_2\right.\\
& + \  2\left.\int_{\mathcal{S}(x)\times\mathcal{S}(x)}[\omega\omega_1\omega_2]^{-1}\delta(k_1-k-k_2)\delta(\omega_1-\omega-\omega_2)[f_1f_2+ff_1]\mathrm{d}k_1\mathrm{d}k_2\right].
\end{aligned}
\end{equation*}
From which, we can use \eqref{Propo:LowerBound:E6}, to bound $f,f_1,f_2$ from below
\begin{equation*}
\begin{aligned}
K[f] \ \ge \ & [\omega_0 + 12]^{-3}\left[\int_{\mathcal{S}(x)\times\mathcal{S}(x)}\delta(k-k_1-k_2)\delta(\omega-\omega_1-\omega_2){f_0^*}^2e^{-2\mathcal{C}(T)t}\mathrm{d}k_1\mathrm{d}k_2\right.\\
& + \  4\left.\int_{\mathcal{S}(x)\times\mathcal{S}(x)}\delta(k_1-k-k_2)\delta(\omega_1-\omega-\omega_2){f_0^*}^2e^{-2\mathcal{C}(T)t}\mathrm{d}k_1\mathrm{d}k_2\right],
\end{aligned}
\end{equation*}
for all $t\in[0,T]$.

The above inequality leads to
\begin{equation}\label{Propo:LowerBound:E7}
\begin{aligned}
K[f] \ \ge \ & \frac{{f_0^*}^2e^{-2\mathcal{C}(T)t}}{[\omega_0 + 12]^3}\left[\int_{\mathcal{S}(x)\times\mathcal{S}(x)}\delta(k-k_1-k_2)\delta(\omega-\omega_1-\omega_2)\mathrm{d}k_1\mathrm{d}k_2\right.\\
& + \  4\left.\int_{\mathcal{S}(x)\times\mathcal{S}(x)}\delta(k_1-k-k_2)\delta(\omega_1-\omega-\omega_2)\mathrm{d}k_1\mathrm{d}k_2\right]\\
 \ \ge \ & \frac{{f_0^*}^2e^{-2\mathcal{C}(T)t}}{[\omega_0 + 12]^3}\mathfrak{C}_2
 \ \ge \  \mathcal{C}_1e^{-2\mathcal{C}(T)t},
\end{aligned}
\end{equation}
for all $t\in[0,T]$.
Note that $\mathcal{C}_1$ is a universal strictly positive constant.

We follow the strategy of \cite{ToanBinh} by plugging \eqref{Propo:LowerBound:E7} into \eqref{Propo:LowerBound:E5} 
\begin{equation}\label{Propo:LowerBound:E8}
\begin{aligned}
f(t,k) \ \ge  \ & {f_0^*}e^{-\mathcal{C}(T)t} \ + \ \mathcal{C}_1 \int_0^te^{-3\mathcal{C}(T)(t-s)}\mathrm{d}s\\
 \ \ge  \ & {f_0^*}e^{-\mathcal{C}(T)t} \ + \   \frac{\mathcal{C}_1}{3\mathcal{C}(T)}[1-e^{-3\mathcal{C}(T)t}],
\end{aligned}
\end{equation}
for all $t\in[0,T]$.

We define the time-dependent function $$F(t)={f_0^*}e^{-\mathcal{C}(T)t} \ + \  \frac{\mathcal{C}_1}{3\mathcal{C}(T)}[1-e^{-3\mathcal{C}(T)t}],$$ which is  continuous and non-negative. 

Pick a finite time $t_0=\frac{c}{\mathcal{C}(T)}>0$, in which $c$ is a fixed constant to be determined later. For $t\in[0,t_0]$, it is clear that $F(t)\ge {f_0^*}e^{-\mathcal{C}(T)t}= {f_0^*}e^{-c}>0$. When $t>t_0$, then $F(t)\ge \frac{\mathcal{C}_1}{3\mathcal{C}(T)} + {f_0^*}e^{-3\mathcal{C}(T)t}[e^{2\mathcal{C}(T)t}-\frac{\mathcal{C}_1}{3\mathcal{C}(T){f_0^*}}]> \frac{\mathcal{C}_1}{3\mathcal{C}(T)} + {f_0^*}e^{-3\mathcal{C}(T)t}[e^{2c}-\frac{\mathcal{C}_1}{3\mathcal{C}(T){f_0^*}}].$ For a suitable choice of $c$, $e^{2c}=\frac{\mathcal{C}_1}{3\mathcal{C}(T){f_0^*}}$. It then follows that $F(t)>\frac{\mathcal{C}_1}{3\mathcal{C}(T)}$, for all $t\in[0,T]$.

As a consequence, $f(t,k)$ is bounded from below by a strictly positive function $\frac{\mathcal{C}_1}{3\mathcal{C}(t)}$  for  $k\in \mathcal{S}(x)$. Since $\mathbb{B}(t)$ is an non-decreasing function of time, it follows that $\frac{\mathcal{C}_1}{3\mathcal{C}(t)}$  is a non-increasing function of time.

\end{proof}

\subsubsection{Convergence to equilibrium of the solution of the equation with the cut-off collision operator on the collisional invariant region $\mathcal{S}(x)$}\label{Sec:Con1}
The below proposition shows the convergence to equilibrium of the equation with cut-off operators. This contains the main ingredients of the proof of the convergence in the non cut-off case.
\begin{proposition}\label{Propo:ConvergenceEqCutoff}
Let $f$ be a positive, classical solution in $C([0,\infty),$ $C^1(\mathcal{S}(x)))$ $\cap$ $ C^1((0,\infty),$ $C^1(\mathcal{S}(x)))$ of \eqref{PhononEqC} on $\mathcal{S}(x)$, with the initial condition $f_0\in C(\mathcal{S}(x))$, $f_0\ge0$. Let $ E_x \in\mathbb{R}_+$ be a  constant  and  
\begin{equation}\label{Propo:ConvergenceEqCutoff:1}
\begin{aligned}
\int_{\mathcal{S}(x)}\frac{1}{a_x   }\mathrm{d}k \ = & \ E_x \ = \ \int_{\mathcal{S}(x)}{\omega(k)}{f_0(k)}\mathrm{d}k,
\end{aligned}
\end{equation}
has a unique solution $a_x\in\mathbb{R}_+$; the local equilibrium on $\mathcal{S}(x)$ can be uniquely determined as
\begin{equation}\label{Propo:ConvergenceEqCutoff:3}
\mathcal{F}^c(k) \ = \ \frac{1}{a_x\omega(k)}.
\end{equation}
 Then, the following limits always hold true, 
 \begin{equation}\label{Propo:ConvergenceEqCutoff:4}
\lim_{t\to\infty}\left\|f(t,\cdot)- \mathcal{F}^c\right\|_{L^1(\mathcal{S}(x))}=0.
\end{equation}
and
\begin{equation}\label{Propo:ConvergenceEqCutoff:5}
\lim_{t\to\infty}\left|\int_{\mathcal{S}(x)}\ln[f]\mathrm{d}k -\int_{\mathcal{S}(x)}\ln\left[\mathcal{F}^c\right]\mathrm{d}k\right|=0.
\end{equation}
If, in addition, there is a positive constant $M^*>0$ such that $f(t,k)<M^*$ for all $t\in[0,\infty)$ and for all $k\in\mathcal{S}(x)$, then 
 \begin{equation}\label{Propo:ConvergenceEqCutoff:6}
\lim_{t\to\infty}\left\|f(t,\cdot)- \mathcal{F}^c\right\|_{L^p(\mathcal{S}(x))}=0,
\ \ \ \ \forall p\in[1,\infty).
\end{equation}
If we suppose further that $f_0(k)>0$ for all $k\in\mathcal{S}(x)$, there exists a constant $M_*$ such that $f(t,k)>M_*$ for all $t\in[0,\infty)$ and for all $k\in\mathcal{S}(x)$.
\end{proposition}
We need the following Lemma, whose proof could be found in the Appendix. 
\begin{lemma}\label{Propo:Entropy2C}
Let $\mathcal{S}(x)$ be a collisonal invariant region and $f$ be a positive function  such that $f\omega\in L^1(\mathcal{S}(x))$. Let 
\begin{equation}
\label{Propo:Entropy2C:EquilibriumC}
\mathcal{F}^c(k)  \ = \ \frac{1}{a_x\omega(k)} \ = : \ \frac{1}{\mathcal{E}^c(k)} ,
\end{equation}
where the constant $a_x\in\mathbb{R}_+$  such that $\mathcal{F}^c(k) >0$ for all $k\in\mathcal{S}(x)$. 

Suppose, in addition, that
\begin{equation}
\label{Propo:Entropy2C:ConservationC}
\int_{\mathcal{S}(x)}f(k)\omega(k)\mathrm{d}k \ = \  \int_{\mathcal{S}(x)}\mathcal{F}^c(k)\omega(k)\mathrm{d}k.
\end{equation}

 Then, the following inequalities always hold true
 \begin{equation}
\label{Propo:Entropy2C:0}
0 \le  \  S_c[\mathcal{F}^c] \ - \ S_c[f],
\end{equation} 
and
\begin{equation}
\label{Propo:Entropy2C:1}
\|f-\mathcal{F}^c\|_{L^1(\mathcal{S}(x))} 
 \lesssim  \  \left[S_c[\mathcal{F}^c] \ - \ S_c[f]\right]^\frac12,
\end{equation} 
in which the constant on the right hand side does not depend on $f$; $S_c[f]$ is defined in \eqref{EntropyC}. 
\end{lemma}
\begin{proof} We divide the proof in to several steps.

{\bf Step 1: Entropy estimates.}
Let us first recall \eqref{TimeDerivativeEntro5C}, which is written as follows
\begin{equation*}
\begin{aligned}
 \  \partial_t\int_{\mathcal{S}(x)}\ln(f)\mathrm{d}k
 \ = & \int_{\mathcal{S}(x)\times\mathcal{S}(x)\times\mathcal{S}(x)}[\omega\omega_1\omega_2]^{-1}\delta(k-k_1-k_2)\delta(\omega-\omega_1-\omega_2)\times\\
 \times  &\frac{\left[g_1+g_2-g\right]^2}{gg_1g_2}\mathrm{d}k\mathrm{d}k_1\mathrm{d}k_2.
\end{aligned}
\end{equation*}
The above identity shows that $\int_{\mathcal{S}(x)}\ln(f)\mathrm{d}k$ is an increasing function of time. In particular $\int_{\mathcal{S}(x)}\ln(f)\mathrm{d}k-\int_{\mathcal{S}(x)}\ln(f_0)\mathrm{d}k\ge0.$
Picking $n\in\mathbb{N}$ and considering the difference of the entropy at two times $n$ and  $n+1$ yields
\begin{equation*}
\begin{aligned}
& \  \left(\int_{\mathcal{S}(x)}\ln(f(2^{n+1},k))\mathrm{d}k -\int_{\mathcal{S}(x)}\ln(f_0(k))\mathrm{d}k\right)- \left(\int_{\mathcal{S}(x)}\ln(f(2^{n},k))\mathrm{d}k -\int_{\mathcal{S}(x)}\ln(f_0(k))\mathrm{d}k\right)\\
 \ = & \int_{2^{n}}^{2^{n+1}}\int_{\mathcal{S}(x)\times\mathcal{S}(x)\times\mathcal{S}(x)}[\omega\omega_1\omega_2]^{-1}\delta(k-k_1-k_2)\delta(\omega-\omega_1-\omega_2)\times\\
 \times  &\frac{\left[g_1+g_2-g\right]^2}{gg_1g_2}\mathrm{d}k\mathrm{d}k_1\mathrm{d}k_2dt.
\end{aligned}
\end{equation*}
Since the quantity $\int_{\mathcal{S}(x)}\ln(f(2^n,k))\mathrm{d}k -\int_{\mathcal{S}(x)}\ln(f_0(k))\mathrm{d}k$ is always positive, we deduce from the above that

\begin{equation*}
\begin{aligned}
& \  \int_{\mathcal{S}(x)}\ln(f(2^{n+1},k))\mathrm{d}k\ - \ \int_{\mathcal{S}(x)}\ln(f_0(k))\mathrm{d}k \ge \\
 \ \ge & \int_{2^{n}}^{2^{n+1}}\int_{\mathcal{S}(x)\times\mathcal{S}(x)\times\mathcal{S}(x)}[\omega\omega_1\omega_2]^{-1}\delta(k-k_1-k_2)\delta(\omega-\omega_1-\omega_2)\frac{\left[g_1+g_2-g\right]^2}{gg_1g_2}\mathrm{d}k\mathrm{d}k_1\mathrm{d}k_2dt.
\end{aligned}
\end{equation*}
By Lemma \ref{Propo:Entropy2C}, applied to the left hand side of the above inequality, we find
\begin{equation}\label{Propo:ConvergenceEqCutoff:E1}
\begin{aligned}
& \  \int_{\mathcal{S}(x)}\ln(\mathcal{F}^c(k))\mathrm{d}k  \ - \ \int_{\mathcal{S}(x)}\ln(f_0(k))\mathrm{d}k  \ge \\
 \ \ge & \int_{2^{n}}^{2^{n+1}}\int_{\mathcal{S}(x)\times\mathcal{S}(x)\times\mathcal{S}(x)}[\omega\omega_1\omega_2]^{-1}\delta(k-k_1-k_2)\delta(\omega-\omega_1-\omega_2)\frac{\left[g_1+g_2-g\right]^2}{gg_1g_2}\mathrm{d}k\mathrm{d}k_1\mathrm{d}k_2dt,
\end{aligned}
\end{equation}
which, after dividing both sides by $2^n$, implies
\begin{equation}\label{Propo:ConvergenceEqCutoff:E2}
\begin{aligned}
& \  \frac{1}{2^n}\left[\int_{\mathcal{S}(x)}\ln(\mathcal{F}^c(k))\mathrm{d}k  \ \ - \ \int_{\mathcal{S}(x)}\ln(f_0(k))\mathrm{d}k\right]  \ge \\
 \ \ge & \frac{1}{2^n}\int_{2^{n}}^{2^{n+1}}\int_{\mathcal{S}(x)\times\mathcal{S}(x)\times\mathcal{S}(x)}[\omega\omega_1\omega_2]^{-1}\delta(k-k_1-k_2)\delta(\omega-\omega_1-\omega_2)\frac{\left[g_1+g_2-g\right]^2}{gg_1g_2}\mathrm{d}k\mathrm{d}k_1\mathrm{d}k_2dt.
\end{aligned}
\end{equation}
As a consequence, there exists a sequence of times $t_n\in[2^n,2^{n+1}]$ such that
\begin{equation}\label{Propo:ConvergenceEqCutoff:E3}
\begin{aligned}
  & \lim_{n\to\infty}\Big[\int_{\mathcal{S}(x)\times\mathcal{S}(x)\times\mathcal{S}(x)}[\omega\omega_1\omega_2]^{-1}\delta(k-k_1-k_2)\delta(\omega-\omega_1-\omega_2)\times\\
 \times  &\frac{\left[g_1(t_n)+g_2(t_n)-g(t_n)\right]^2}{g(t_n)g_1(t_n)g_2(t_n)}\mathrm{d}k\mathrm{d}k_1\mathrm{d}k_2\Big] \ = \ 0.
\end{aligned}
\end{equation}
For the sake of simplicity, we denote $g(t_n)$ and $f(t_n)$ by $g^n$ and $f^n$.

{\bf Step 2: The convergence.} 
%


Taking advantage of the fact $g^n\le 2N$ in the cut-off region of the operator $\chi_N^*$, the following limit can be deduced from \eqref{Propo:ConvergenceEqCutoff:E3}
\begin{equation}\label{Propo:ConvergenceEqCutoff:E8}
\begin{aligned}
  & \lim_{n\to\infty}\Big[\int_{\mathcal{S}(x)\times\mathcal{S}(x)\times\mathcal{S}(x)}[\omega\omega_1\omega_2]^{-1}\delta(k-k_1-k_2)\delta(\omega-\omega_1-\omega_2)\chi_N^*\times\\
 \times  &{\left[g_1^n+g_2^n-g^n\right]^2}\mathrm{d}k\mathrm{d}k_1\mathrm{d}k_2\Big] \ = \ 0,
\end{aligned}
\end{equation}
in which the product $g^ng_1^ng_2^n$ has been eliminated. Since $g^ng_1^ng_2^n$ is removed, the inequality \eqref{Propo:EntropyC:0} can be applied, leading to another limit
 \begin{equation}\label{Propo:ConvergenceEqCutoff:E9}
\begin{aligned}
& \lim_{n\to\infty}\int_{\mathcal{S}(x)}\sqrt{f^n\left|\mathbb{Q}_c^{N,+}[g^n]-\mathbb{Q}_c^{N,-}[g^n]\right|}\mathrm{d}k \ =  \ 0.
\end{aligned}
\end{equation}
The above expression contains $f^n$, which can be, again, eliminated using the lower bound $f^n\ge \frac{1}{2N}$ in the cut-off region, yielding
\begin{equation}\label{Propo:ConvergenceEqCutoff:E10}
\begin{aligned}
& \lim_{n\to\infty}\int_{\mathcal{S}(x)}\sqrt{\left|\mathbb{Q}_c^{N,+}[g^n]-\mathbb{Q}_c^{N,-}[g^n]\right|}\mathrm{d}k \ =  \ 0.
\end{aligned}
\end{equation}
Replacing $\mathbb{Q}_c^{N,+}[g^n] =  g^n\mathbb{L}^N_c[g^n]$ in the above formula leads to
\begin{equation}\label{Propo:ConvergenceEqCutoff:E11}
\begin{aligned}
& \lim_{n\to\infty}\int_{\mathcal{S}(x)}\sqrt{\left|g^n\mathbb{L}^N_c-\mathbb{Q}_c^{N,-}[g^n]\right|}\mathrm{d}k \ =  \ 0.
\end{aligned}
\end{equation}
Notice that $g^n\mathbb{L}^N_c=  g^n\chi_N[g^n]\tilde{\mathbb{L}}^N_c,$ in which $\tilde{\mathbb{L}}^N_c$ takes the following form
\begin{equation}\label{Propo:ConvergenceEqCutoff:FixedE1}
\begin{aligned}
\tilde{\mathbb{L}}^N_c    \ := \ & \mathcal{G}^N_1[g^n] \ + \  \mathcal{G}^N_2[g^n]\\
 \ := \ & \int_{\mathcal{S}(x)\times \mathcal{S}(x)}\chi_N[g^n(k_1)]\chi_N[g^n(k_2)]\delta(k-k_1-k_2)\delta(\omega-\omega_1-\omega_2)\mathrm{d}k_1\mathrm{d}k_2\\
 &  + 2\int_{\mathcal{S}(x)\times \mathcal{S}(x)}\chi_N[g^n(k_1)]\chi_N[g^n(k_2)]\delta(k_1-k-k_2)\delta(\omega_1-\omega-\omega_2)\mathrm{d}k_1\mathrm{d}k_2.
\end{aligned}
\end{equation}
Let us consider the first sequence $\{\mathcal{G}^N_1[g^n]\}$. We will show that this sequence is equicontinuous in all $L^p(\mathcal{S}(x))$ with $1\le p<\infty$.  This, by the Kolmogorov-Riesz theorem \cite{hanche2010kolmogorov} implies the strong convergence of $\{\mathcal{G}^N_1[g^n]\}$ towards a function $\mathcal{G}_1$ in  $L^p(\mathcal{S}(x))$ with $1\le p<\infty$. To see this, let us consider any vector $k'$ belonging to a ball $B(O,\delta)$ centered at the origin and with radius $\delta$, and estimate the difference $ \mathcal{G}^N_1[g^n](\cdot+k')- \mathcal{G}^N_1[g^n](\cdot)$ in the $L^p$-norm 
\begin{equation}
\label{Propo:ConvergenceEqCutoff:FixedE2}
\begin{aligned}
&\int_{\mathcal{S}(x)}|\mathcal{G}^N_1[g^n](k+k')-\mathcal{G}^N_1[g^n](k)|^p\mathrm{d}k\\ 
\ = \ &  \int_{\mathcal{S}(x)}\Big|\int_{\mathcal{S}(x)}\Big[\chi_N[g^n(k'+k-k_1)]\delta(\omega(k')-\omega(k_1)-\omega(k'+k-k_1))-\\
&- \chi_N[g^n(k-k_1)]\delta(\omega(k)-\omega(k_1)-\omega(k-k_1))\Big]\chi_N[g^n(k_1)]\mathrm{d}k_1\Big|^p\mathrm{d}k.
\end{aligned}
\end{equation}
To estimate the above quantity, we will use the triangle inequality, as follows
\begin{equation}
\label{Propo:ConvergenceEqCutoff:FixedE3}
\begin{aligned}
&\int_{\mathcal{S}(x)}|\mathcal{G}^N_1[g^n](k+k')-\mathcal{G}^N_1[g^n](k)|^p\mathrm{d}k\\ 
\ \lesssim \ &  \int_{\mathcal{S}(x)}\Big|\int_{\mathcal{S}(x)}|\chi_N[g^n(k'+k-k_1)]-\chi_N[g^n( k-k_1)]|\times\\
&\indent\times\delta(\omega(k'+k)-\omega(k_1)-\omega(k'+k-k_1))\chi_N[g^n(k_1)]\mathrm{d}k_1\\
&+\int_{\mathcal{S}(x)}\chi_N[g^n(k-k_1)]|\delta(\omega(k'+k)-\omega(k_1)-\omega(k'+k-k_1))\\
&\indent -\delta(\omega(k)-\omega(k_1)-\omega(k-k_1))|\chi_N[g^n(k_1)]\mathrm{d}k_1\Big|^p\mathrm{d}k.
\end{aligned}
\end{equation}
In the right hand side of this equality, we have the sum of two integrals inside the power of order $p$. To facilitate the computations, we use Young's inequality to split this into two separate integrals as
\begin{equation}
\label{Propo:ConvergenceEqCutoff:FixedE4}
\begin{aligned}
&\int_{\mathcal{S}(x)}|\mathcal{G}^N_1[g^n](k+k')-\mathcal{G}^N_1[g^n](k)|^p\mathrm{d}k\\ 
\ \lesssim \ &  \int_{\mathcal{S}(x)}\Big|\int_{\mathcal{S}(x)}|\chi_N[g^n(k'+k-k_1)]-\chi_N[g^n(k-k_1)]|\times\\
& \times\delta(\omega(k'+k)-\omega(k_1)-\omega(k'+k-k_1))\chi_N[g^n(k_1)]\mathrm{d}k_1\Big|^p\mathrm{d}k\\
&+\int_{\mathcal{S}(x)}\Big|\int_{\mathcal{S}(x)}\chi_N[g^n(k-k_1)]|\delta(\omega(k'+k)-\omega(k_1)-\omega(k'+k-k_1))\\
&\indent -\delta(\omega(k)-\omega(k_1)-\omega(k-k_1))|\chi_N[g^n(k_1)]\mathrm{d}k_1\Big|^p\mathrm{d}k.
\end{aligned}
\end{equation}
We can choose $\delta$ small such that $\chi_N[g^n(k'+k-k_1)]-\chi_N[g^n(k-k_1)]$ is small, uniformly in $k$ and $k_1$, thanks to the cut-off property $\frac{1}{N}\le |f^n(k)|, |\nabla f^n(k)|\le N$ in the cut-off region. Combining this observation, with Proposition \ref{Lemma:RestrictionSx}, Corollary \ref{Coll:Edges} and the boundedness of $\chi_N[g^n(k_1)]$, we can choose $\delta$ small enough, depending on a small $\epsilon>0$, such that the first term on the right hand side is smaller than $\epsilon^p/2$. The second term on the right hand side can also be bounded by $\epsilon^p/2$ using Proposition \ref{Propo:Lip} and the fact that $\chi_N[g^n(k-k_1)]$ and $\chi_N[g^n(k_1)]$ are both bounded by $1$. As a result, for any small constant $\epsilon>0$, we can choose $\delta$ such that for any $k'\in B(O,\delta)$, 
\begin{equation}
\label{Propo:ConvergenceEqCutoff:FixedE5}
\int_{\mathcal{S}(x)}|\mathcal{G}^N_1[g^n](k+k')-\mathcal{G}^N_1[g^n](k)|^p\mathrm{d}k 
\ \lesssim \ \epsilon^p,
\end{equation}
which shows that the sequence $\mathcal{G}^N_1[g^n]$ is indeed equicontinuous in $L^p(\mathcal{S}(x))$ and the existence of $\sigma_1\in L^p(\mathcal{S}(x))$ satisfying $\lim_{n\to\infty}\mathcal{G}^N_1[g^n]=\sigma_1$ in $L^p(\mathcal{S}(x))$ for all $p\in[1,\infty)$ is guaranteed by the Kolmogorov-Riesz theorem \cite{hanche2010kolmogorov}.

The same argument can be applied to $\mathcal{G}^N_2[g^n]$, leading to  the existence of $\sigma_2\in L^p(\mathcal{S}(x))$ satisfying $\lim_{n\to\infty}\mathcal{G}^N_2[g^n]=\sigma_2$  in $L^p(\mathcal{S}(x))$ for all $p\in[1,\infty)$  by the Kolmogorov-Riesz theorem \cite{hanche2010kolmogorov}.  As a result $\lim_{n\to\infty} \tilde{\mathbb{L}}_c^N = \sigma = \sigma_1 
+ \sigma_2$  in $L^p(\mathcal{S}(x))$ for all $p\in[1,\infty)$. 

Similarly, if we define
\begin{equation}\label{Propo:ConvergenceEqCutoff:FixedE6}
\begin{aligned}
&\tilde{\mathbb{Q}}_c^{N,-}[g](k) \ = \ \tilde{\mathbb{Q}}_c^{N,-,1}[g](k)+\tilde{\mathbb{Q}}_c^{N,-,2}[g](k)+\tilde{\mathbb{Q}}_c^{N,-,3}[g](k) \ :=  \\
 \ =  & \ 2\int_{\mathcal{S}(x)\times \mathcal{S}(x)}\chi_N[1/g](k_1)\chi_N[1/g](k_2)[\omega\omega_1\omega_2]^{-1}\delta(k-k_1-k_2)\delta(\omega-\omega_1-\omega_2){g_1}\mathrm{d}k_1\mathrm{d}k_2\\
& \ +2\int_{\mathcal{S}(x)\times \mathcal{S}(x)}\chi_N[1/g](k_1)\chi_N[1/g](k_2)[\omega\omega_1\omega_2]^{-1}\delta(k_1-k-k_2)\delta(\omega_1-\omega-\omega_2){g_1}\mathrm{d}k_1\mathrm{d}k_2\\
& \ -2\int_{\mathcal{S}(x)\times \mathcal{S}(x)}\chi_N[1/g](k_1)\chi_N[1/g](k_2)[\omega\omega_1\omega_2]^{-1}\delta(k_1-k-k_2)\delta(\omega_1-\omega-\omega_2){g_2}\mathrm{d}k_1\mathrm{d}k_2,
\end{aligned}
\end{equation}
the Kolmogorov-Riesz theorem \cite{hanche2010kolmogorov} can  be used in the same manner to deduce the existence of a function $\varsigma$ such that we also have $\lim_{n\to \infty}\tilde{\mathbb{Q}}_c^{N,-}[g^n] \ = \ \varsigma$ in $L^p(\mathcal{S}(x))$ for all $p\in[1,\infty)$.

Now, the fact that $\lim_{n\to \infty}\tilde{\mathbb{Q}}_c^{N,-}[g^n] \ = \ \varsigma$  and $\lim_{n\to\infty} \tilde{\mathbb{L}}_c^N = \sigma $ can be used to replace the quantity $\mathbb{Q}_c^{N,-}[g^n]$ by $\varsigma$ and the quantity $\tilde{\mathbb{L}}_c^N$ by $\sigma$  in \eqref{Propo:ConvergenceEqCutoff:E9} and  \eqref{Propo:ConvergenceEqCutoff:E11} to have

\begin{equation}\label{Propo:ConvergenceEqCutoff:E12}
\begin{aligned}
& \lim_{n\to\infty}\int_{\mathcal{S}(x)}\sqrt{\left|\sigma\chi_N[f^n]-f^n\chi_N[f^n]\varsigma\right|}\mathrm{d}k \ =  \ 0,
\end{aligned}
\end{equation} 
and 
\begin{equation}\label{Propo:ConvergenceEqCutoff:E12}
\begin{aligned}
& \lim_{n\to\infty}\int_{\mathcal{S}(x)}\sqrt{\left|g^n\chi_N[g^n]\sigma-\varsigma\chi_N[f^n]\right|}\mathrm{d}k \ =  \ 0.
\end{aligned}
\end{equation}
Due to its boundedness,  the sequences $\{g^n\chi_N[f^n]\}$, $\{f^n\chi_N[f^n]\}$ and $\{\chi_N[f^n]\}$ converge weakly to $g^\infty_N$, $f^\infty_N$ and $\xi_N^\infty$ in $L^1(\mathcal{S}(x))$, it follows immediately that $g^\infty_N\sigma=\xi_N^\infty\varsigma$  and $\xi^\infty_N\sigma=f^\infty_N\varsigma$.

By a similar argument as above, $\{\chi_N[f^n]\}$ is also equicontinuous in $L^p(\mathcal{S}(x))$ and then $\lim_{n\to\infty}\chi_N[f^n]=\xi_N^\infty$ in $L^p(\mathcal{S}(x))$ for all $p\in[1,\infty)$ by the Kolmogorov-Riesz theorem \cite{hanche2010kolmogorov}.  As a consequence,
\begin{equation*}
\begin{aligned}
\varsigma(k) \ = 
 \   & \ 2\int_{\mathcal{S}(x)\times \mathcal{S}(x)}\xi_N^\infty(k_1)\xi_N^\infty(k_2)[\omega\omega_1\omega_2]^{-1}\delta(k-k_1-k_2)\delta(\omega-\omega_1-\omega_2){g^\infty_N(k_1)}\mathrm{d}k_1\mathrm{d}k_2\\
& \ +2\int_{\mathcal{S}(x)\times \mathcal{S}(x)}\xi_N^\infty(k_1)\xi_N^\infty(k_2)[\omega\omega_1\omega_2]^{-1}\delta(k_1-k-k_2)\delta(\omega_1-\omega-\omega_2)g^\infty_N(k_1)\mathrm{d}k_1\mathrm{d}k_2\\
& \ -2\int_{\mathcal{S}(x)\times \mathcal{S}(x)}\xi_N^\infty(k_1)\xi_N^\infty(k_2)[\omega\omega_1\omega_2]^{-1}\delta(k_1-k-k_2)\delta(\omega_1-\omega-\omega_2)g^\infty_N(k_2)\mathrm{d}k_1\mathrm{d}k_2,
\end{aligned}
\end{equation*}
and
\begin{equation*}
\begin{aligned}
\sigma(k)   \ = \ & \int_{\mathcal{S}(x)\times \mathcal{S}(x)}\xi_N^\infty(k_1)\xi_N^\infty(k_2)\delta(k-k_1-k_2)\delta(\omega-\omega_1-\omega_2)\mathrm{d}k_1\mathrm{d}k_2\\
 &  + 2\int_{\mathcal{S}(x)\times \mathcal{S}(x)}\xi_N^\infty(k_1)\xi_N^\infty(k_2)\delta(k_1-k-k_2)\delta(\omega_1-\omega-\omega_2)\mathrm{d}k_1\mathrm{d}k_2,
\end{aligned}
\end{equation*}
which can be combined with \eqref{Propo:ConvergenceEqCutoff:E12} and the fact that $\{g^n\chi_N[f^n]\}$, $\{f^n\chi_N[f^n]\}$ converge weakly to $g^\infty_N$, $f^\infty_N$  to give  
\begin{equation}\label{Propo:ConvergenceEqCutoff:E12a}
\begin{aligned}
  & \int_{\mathcal{S}(x)\times \mathcal{S}(x)}g^\infty_N(k)\xi_N^\infty(k)\xi_N^\infty(k_1)\xi_N^\infty(k_2)\delta(k-k_1-k_2)\delta(\omega-\omega_1-\omega_2)\mathrm{d}k_1\mathrm{d}k_2\\
 &  + 2\int_{\mathcal{S}(x)\times \mathcal{S}(x)}g^\infty_N(k)\xi_N^\infty(k)\xi_N^\infty(k_1)\xi_N^\infty(k_2)\delta(k_1-k-k_2)\delta(\omega_1-\omega-\omega_2)\mathrm{d}k_1\mathrm{d}k_2\\
\ = 
 \   & \ 2\int_{\mathcal{S}(x)\times \mathcal{S}(x)}\xi_N^\infty(k)\xi_N^\infty(k_1)\xi_N^\infty(k_2)[\omega\omega_1\omega_2]^{-1}\delta(k-k_1-k_2)\delta(\omega-\omega_1-\omega_2){g^\infty_N(k_1)}\mathrm{d}k_1\mathrm{d}k_2\\
& \ +2\int_{\mathcal{S}(x)\times \mathcal{S}(x)}\xi_N^\infty(k)\xi_N^\infty(k_1)\xi_N^\infty(k_2)[\omega\omega_1\omega_2]^{-1}\delta(k_1-k-k_2)\delta(\omega_1-\omega-\omega_2)g^\infty_N(k_1)\mathrm{d}k_1\mathrm{d}k_2\\
& \ -2\int_{\mathcal{S}(x)\times \mathcal{S}(x)}\xi_N^\infty(k)\xi_N^\infty(k_1)\xi_N^\infty(k_2)[\omega\omega_1\omega_2]^{-1}\delta(k_1-k-k_2)\delta(\omega_1-\omega-\omega_2)g^\infty_N(k_2)\mathrm{d}k_1\mathrm{d}k_2,
\end{aligned}
\end{equation}
for a.e. $k$ in $\mathcal{S}(x)$. 

From \eqref{Propo:ConvergenceEqCutoff:E12a}, we deduce that $$g^\infty_N(k)\xi_N^\infty(k)=g^\infty_N(k_1)\xi_N^\infty(k_1)+g^\infty_N(k_2)\xi_N^\infty(k_2),$$ when $k=k_1+k_2$ and $\omega(k)=\omega(k_1)+\omega(k_2)$, for a.e. $k$ in $\mathcal{S}(x)$.
The proofs of Proposition \ref{Propo:CollsionInvariant} and Lemma \ref{Lemma:EquilibriumCutoff} can then be redone, yielding $g^\infty_N(k) \xi_N^\infty(k)= A_N \omega(k) =:\mathcal{E}^c(k)> 0$ for some  constant $A_N\in\mathbb{R}$. These constants are subjected to the conservation of energy 
\begin{equation}
\label{Propo:ConvergenceEqCutoff:E13}
\begin{aligned}
\int_{\mathcal{S}(x)}\frac{\omega(k) }{A_N \omega(k)} \mathrm{d}k\ = & \ \lim_{n\to\infty}\int_{\mathcal{S}(x)}\omega(k)  f^n\chi_N[f^n]\mathrm{d}k
{\ =: \ E^N_x.}
\end{aligned}
\end{equation}
In addition, we have $f_N^\infty = \frac{1}{A_N \omega(k) }$.
Since  $\lim_{N\to\infty}E_x^N=E_x$, when $N$ is large enough $\frac{1}{N}<g^\infty_N(k), f_N^\infty(k)<N$ for all $k\in\mathcal{S}(x)$. As a consequence, $g^n$ and $f^n$ converge almost everywhere to $g^\infty_N(k)$, and $f_N^\infty(k)$. 

The fact that $f^n$ converges to $f_N^\infty(k)$ almost everywhere, when $N$ is sufficiently large, ensures the existence of $N_0>0$ such that $f_N^\infty(k)=f_M^\infty(k)$ for all $N,M>N_0$. Passing to the limits $N\to\infty$ in \eqref{Propo:ConvergenceEqCutoff:E13}, we find $A_N=A$  for all $N>N_0$, with
\begin{equation}
\label{Propo:ConvergenceEqCutoff:E13}
\begin{aligned}
\int_{\mathcal{S}(x)}\frac{\omega(k) }{A \omega(k) } \mathrm{d}k\ = & \ E_x.
\end{aligned}
\end{equation}
As a result,
 $$\lim_{n\to\infty}f^n(k)=\frac{1}{A \omega(k) }=:\mathcal{F}^c$$  almost everywhere on $\mathcal{S}(x)$, which then implies 
 $$\liminf_{n\to\infty}\int_{\mathcal{S}(x)}\ln[f]\mathrm{d}k \ge \int_{\mathcal{S}(x)}\ln[\mathcal{F}^c]\mathrm{d}k,$$
 by Fatou's Lemma. Therefore, due to Lemma \ref{Propo:Entropy2C}
$$\lim_{n\to\infty}[S_c[\mathcal{F}^c] \ - \ S_c[f^n]]=0,$$
leading to
  $$\lim_{t\to\infty}[S_c[\mathcal{F}^c] \ - \ S_c[f(t)]]=0.$$
 By \eqref{Propo:Entropy2C:1}, we finally obtain
$$\lim_{t\to\infty}
\|f-\mathcal{F}^c\|_{L^1(\mathcal{S}(x))} =0.$$
{\bf Step 3: Additional assumption $f(t,k)<M^*$ for all $t\in[0,\infty)$ and for all $k\in\mathcal{S}(x)$.} 
Suppose, in addition, that $f(t,k)<M^*$ for all $t\in[0,\infty)$. By Egorov's theorem, for all $\delta>0$, there exists a set $\mathcal{V}_\delta$, whose measure $m(\mathcal{V}_\delta)$ is smaller than $\delta$ and $f^n$ converges uniformly to $f^\infty(k)$ on $\mathcal{S}(x)\backslash  \mathcal{V}_\delta$. Since $\frac{1}{N}<f^\infty_N(k)<N$, there exists an integer $n_\delta$ such that for all $n>n_\delta$, the inequality $\frac{1}{N}<f^n(k)<N$ holds true for all $k\in \mathcal{S}(x)\backslash  \mathcal{V}_\delta$. As a consequence, for each $\epsilon>0$
$$\|f-\mathcal{F}^c\|_{L^p(\mathcal{S}(x))} \le C \|f-\mathcal{F}^c\|_{L^\infty(\mathcal{S}(x)\backslash  \mathcal{V}_\delta)}  \ + \  C m(\mathcal{V}_\delta)^\frac{1}{p} \ \le C \|f-\mathcal{F}^c\|_{L^\infty(\mathcal{S}(x)\backslash  \mathcal{V}_\delta)}  \ + \  C\delta^\frac{1}{p},$$
where $C$ is a universal constant, for all $1<p<\infty$. 

For any $\epsilon>0$, we can choose $\delta>0$ and a time $t_\delta$ such that for $t>t_\delta$, $ C\delta^\frac{1}{p}<\epsilon/2$ and $C \|f-\mathcal{F}^c\|_{L^\infty(\mathcal{S}(x)\backslash  \mathcal{V}_\delta)}<\epsilon/2$. That implies the strong convergence of $f$ towards $\mathcal{F}^c$ in $L^p(\mathcal{S}(x)$ for all $1<p<\infty$. 

Now, if $f_0(k)>0$ for all $k\in\mathcal{S}(x)$ and $f(t,k)<M^*$ for all $t\in[0,\infty)$ and for all $k\in\mathcal{S}(x)$, by Proposition \ref{Propo:LowerBound}, there exists a constant $M_*$ such that $f(t,k)>M_*$ for all $t\in[0,\infty)$ and for all $k\in\mathcal{S}(x)$.

 \end{proof} 

\subsection{Proof of Theorem \ref{Thm:MainClassical}}\label{Sec:Con2}
The proof of Theorem \ref{Thm:MainClassical} follows from Proposition \ref{Propo:ConvergenceEqCutoff} and Proposition \ref{NoCollisionEffect}.

\section{Appendix}

\subsection{Appendix A: Proof of Lemma \ref{Propo:Entropy2C}}

Define the functional
\begin{equation*}
\Psi_t(f,\mathcal{F}^c) \ = \ [\mathcal{F}^c \ +  \ t(f-\mathcal{F}^c)]^2.
\end{equation*}
It follows from the mean value theorem that
\begin{equation*}
0 \ \le \ \int_0^1\frac{(1-t)(f-\mathcal{F}^c)^2}{\Psi_t(f,\mathcal{F}^c)}\mathrm{d}t  \ = \ s_c[\mathcal{F}^c] \ - \ s_c[f] \ +  \ s'_c[\mathcal{F}^c](f-\mathcal{F}^c).
\end{equation*}
Since $s'(y)=1/y$, we find $s'[\mathcal{F}^c(k)]=a_x\omega(k).$ That leads to
\begin{equation*}
0 \ \le \ \int_0^1\frac{(1-t)(f-\mathcal{F}^c)^2}{\Psi_t(f,\mathcal{F}^c)}\mathrm{d}t \ = \ s_c[\mathcal{F}^c] \ - \ s_c[f] \ +  \ (a_x\omega(k))(f-\mathcal{F}^c).
\end{equation*}

Integrating both sides of the above inequality on $\mathcal{S}(x)$ yields
\begin{equation*}
\begin{aligned}
0 \ \le & \ \int_{\mathcal{S}(x)}\int_0^1\frac{(1-t)(f-\mathcal{F}^c)^2}{\Psi_t(f,\mathcal{F}^c)}\mathrm{d}t \mathrm{d}k\\
 = & \  \int_{\mathcal{S}(x)}s_c[\mathcal{F}^c]\mathrm{d}k \ - \ \int_{\mathcal{S}(x)}s_c[f]\mathrm{d}k \ +  \ \int_{\mathcal{S}(x)}(a_x\omega(k))(f-\mathcal{F}^c)\mathrm{d}k,
\end{aligned}
\end{equation*}
which, by the fact that $$\int_{\mathcal{S}(x)}(a_x\omega(k))(f-\mathcal{F}^c)\mathrm{d}k=0,$$ implies
\begin{equation}\label{Propro2C:e1a}
0 \ \le \ \int_{\mathcal{S}(x)}\int_0^1\frac{(1-t)(f-\mathcal{F}^c)^2}{\Psi_t(f,\mathcal{F}^c)}\mathrm{d}t \mathrm{d}k \ \le \ S_c[\mathcal{F}^c] \ - \ S_c[f].
\end{equation}
Observing that
\begin{equation*}
\begin{aligned}
(\mathcal{F}^c-f)_+ \ = \ & 2\int_{0}^1\frac{\sqrt{1-t}(\mathcal{F}^c-f)_+}{\sqrt{\Psi_t(f,\mathcal{F}^c)}}\sqrt{(1-t)\Psi_t(f,\mathcal{F}^c)}\mathrm{d}t,
\end{aligned}
\end{equation*} 
and applying H\"older's inequality to the right hand side, we obtain the following inequality
\begin{equation*}
\begin{aligned}
(\mathcal{F}^c-f)_+ \ \le \ & 2\left[\int_{0}^1\frac{{(1-t)}(\mathcal{F}^c-f)^2}{{\Psi_t(f,\mathcal{F}^c)}}\mathrm{d}t\right]^\frac12\left[\int_{0}^1{(1-t)\Psi_t(f,\mathcal{F}^c)}\mathrm{d}t\right]^\frac12.
\end{aligned}
\end{equation*} 
Now, observe that for $k\in\mathcal{S}(x)$ satisfying $\mathcal{F}^c(k)>f(k),$ then $$0<\Psi_t(f,\mathcal{F}^c)(k)\le [\mathcal{F}^c(k)]^2$$ for all $t\in[0,1]$. This fact can reduce the above inequality to
\begin{equation*}
\begin{aligned}
(\mathcal{F}^c-f)_+  \ \le \ &2\left[\int_{0}^1\frac{{(1-t)}(\mathcal{F}^c-f)^2}{{\Psi_t(f,\mathcal{F}^c)}}\mathrm{d}t\right]^\frac12\left[\int_{0}^1{(1-t)[\mathcal{F}^c(k)]^2}\mathrm{d}t\right]^\frac12,
\end{aligned}
\end{equation*} 
which, by integrating in $k$
\begin{equation*}
\begin{aligned}
\int_{\mathcal{S}(x)}(\mathcal{F}^c-f)_+ \mathrm{d}k \ \le \ & 2\int_{\mathcal{S}(x)}\left[\int_{0}^1\frac{{(1-t)}(\mathcal{F}^c-f)^2}{{\Psi_t(f,\mathcal{F}^c)}}\mathrm{d}t\right]^\frac12\left[\int_{0}^1{(1-t)[\mathcal{F}^c(k)]^2}\mathrm{d}t\right]^\frac12\mathrm{d}k,
\end{aligned}
\end{equation*} 
and applying H\"older's inequality to the right hand side, gives
\begin{equation*}
\begin{aligned}
\int_{\mathcal{S}(x)}(\mathcal{F}^c-f)_+ \mathrm{d}k \ \le \ &2\left[ \int_{\mathcal{S}(x)}\int_{0}^1\frac{{(1-t)}(\mathcal{F}^c-f)^2}{{\Psi_t(f,\mathcal{F}^c)}}\mathrm{d}t\mathrm{d}k\right]^\frac12\left[\int_{\mathcal{S}(x)}\int_{0}^1{(1-t)[\mathcal{F}^c(k)]^2}\mathrm{d}t\mathrm{d}k\right]^\frac12.
\end{aligned}
\end{equation*} 
Indeed, the second term with the bracket on the right hand side can be computed explicitly, that implies
\begin{equation*}
\begin{aligned}
\int_{\mathcal{S}(x)}(\mathcal{F}^c-f)_+ \mathrm{d}k \ \lesssim \ &\left[ \int_{\mathcal{S}(x)}\int_{0}^1\frac{{(1-t)}(\mathcal{F}^c-f)^2}{{\Psi_t(f,\mathcal{F}^c)}}\mathrm{d}t\mathrm{d}k\right]^\frac12.
\end{aligned}
\end{equation*} 
The above inequality can be combined with \eqref{Propro2C:e1a} to become
\begin{equation*}
\int_{\mathcal{S}(x)}(\mathcal{F}^c-f)_+ \mathrm{d}k \ \lesssim \ \left[S_c[\mathcal{F}^c] \ - \ S_c[f]\right]^\frac12.
\end{equation*} 
Using the boundedness of the dispersion relation $\omega(k)$, we find
\begin{equation*}
\int_{\mathcal{S}(x)}(\mathcal{F}^c-f)_+\omega(k) \mathrm{d}k \ \lesssim \ \int_{\mathcal{S}(x)}(\mathcal{F}^c-f)_+ \mathrm{d}k \ \lesssim \ \left[S_c[\mathcal{F}^c] \ - \ S_c[f]\right]^\frac12.
\end{equation*} 
Now, from the identity 
\begin{equation*}
|f-\mathcal{F}^c| \ = \ f-\mathcal{F}^c \ + \ 2 (\mathcal{F}-f)_+,
\end{equation*}
the above gives
\begin{equation*}\begin{aligned}
\int_{\mathcal{S}(x)}|f-\mathcal{F}^c|\omega(k) \mathrm{d}k \ = & \ \int_{\mathbb{T}^3}(f-\mathcal{F}^c)\omega(k) \mathrm{d}k \ + \ \int_{\mathcal{S}(x)}2 (\mathcal{F}^c-f)_+\omega(k) \mathrm{d}k \\
 \lesssim & \ \int_{\mathcal{S}(x)}(f-\mathcal{F}^c)\omega(k) \mathrm{d}k \ + \ 2\left[S_c[\mathcal{F}^c] \ - \ S_c[f]\right]^\frac12.
\end{aligned}
\end{equation*} 
From the hypothesis $$\int_{\mathcal{S}(x)}(f-\mathcal{F}^c)\omega(k) \mathrm{d}k=0,$$ we then infer from the above inequality that
\begin{equation*}\begin{aligned}
\int_{\mathcal{S}(x)}|f-\mathcal{F}^c|\omega(k) \mathrm{d}k 
 \lesssim & \  \left[S_c[\mathcal{F}^c] \ - \ S_c[f]\right]^\frac12.
\end{aligned}
\end{equation*} 
Using the fact that $\omega(k)\ge \omega_0$, we obtain
\begin{equation*}\begin{aligned}
\int_{\mathcal{S}(x)}|f-\mathcal{F}^c|\mathrm{d}k 
 \lesssim & \  \left[S_c[\mathcal{F}^c] \ - \ S_c[f]\right]^\frac12.
\end{aligned}
\end{equation*}

\bibliographystyle{plain}

\bibliography{WaveTurbulence}

\def\cprime{$'$} \def\cprime{$'$} \def\cprime{$'$} \def\cprime{$'$}
  \def\cprime{$'$} \def\cprime{$'$} \def\cprime{$'$} \def\cprime{$'$}
  \def\cprime{$'$} \def\cprime{$'$}
\begin{thebibliography}{10}

\bibitem{AlonsoGambaBinh}
R.~Alonso, I.~M. Gamba, and M.-B. Tran.
\newblock The {C}auchy problem and {BEC} stability for the quantum
  {B}oltzmann-{G}ross-{P}itaevskii system for bosons at very low temperature.
\newblock {\em arXiv preprint arXiv:1609.07467}, 2016.

\bibitem{ampatzoglou2024inhomogeneous}
I.~Ampatzoglou, J.~K. Miller, N.~Pavlovi{\'c}, and M.~Taskovi{\'c}.
\newblock Inhomogeneous wave kinetic equation and its hierarchy in polynomially
  weighted linfty spaces.
\newblock {\em arXiv preprint arXiv:2405.03984}, 2024.

\bibitem{aoki2002nonequilibrium}
K.~Aoki and D.~Kusnezov.
\newblock Nonequilibrium statistical mechanics of classical lattice varphi4
  field theory.
\newblock {\em Annals of Physics}, 295(1):50--80, 2002.

\bibitem{basile2016thermal}
G.~Basile, C.~Bernardin, M.~Jara, T.~Komorowski, and S.~Olla.
\newblock Thermal conductivity in harmonic lattices with random collisions.
\newblock In {\em Thermal transport in low dimensions}, pages 215--237.
  Springer, 2016.

\bibitem{benney1967propagation}
D.~J. Benney and A.~C. Newell.
\newblock The propagation of nonlinear wave envelopes.
\newblock {\em Journal of mathematics and Physics}, 46(1-4):133--139, 1967.

\bibitem{benney1969random}
D.~J. Benney and A.~C. Newell.
\newblock Random wave closures.
\newblock {\em Studies in Applied Mathematics}, 48(1):29--53, 1969.

\bibitem{benney1966nonlinear}
D.~J. Benney and P.~G. Saffman.
\newblock Nonlinear interactions of random waves in a dispersive medium.
\newblock {\em Proc. R. Soc. Lond. A}, 289(1418):301--320, 1966.

\bibitem{bretherton1964resonant}
F.~P. Bretherton.
\newblock Resonant interactions between waves. the case of discrete
  oscillations.
\newblock {\em Journal of Fluid Mechanics}, 20(3):457--479, 1964.

\bibitem{burq1993controle}
N.~Burq.
\newblock {\em Contr{\^o}le de l'{\'e}quation des plaques en pr{\'e}sence
  d'obstacles strictement convexes}.
\newblock Soci{\'e}t{\'e} math{\'e}matique de France, 1993.

\bibitem{cercignani1999relativistic}
C.~Cercignani and G.~M. Kremer.
\newblock On relativistic collisional invariants.
\newblock {\em Journal of statistical physics}, 96:439--445, 1999.

\bibitem{collot2024stability}
C.~Collot, H.~Dietert, and P.~Germain.
\newblock Stability and cascades for the {K}olmogorov--{Z}akharov spectrum of
  wave turbulence.
\newblock {\em Archive for Rational Mechanics and Analysis}, 248(1):7, 2024.

\bibitem{CraciunBinh}
G.~Craciun and M.-B. Tran.
\newblock A reaction network approach to the convergence to equilibrium of
  quantum boltzmann equations for bose gases.
\newblock {\em ESAIM: Control, Optimisation and Calculus of Variations}, 27:83,
  2021.

\bibitem{dolce2024convergence}
M.~Dolce and R.~Grande.
\newblock On the convergence rates of discrete solutions to the wave kinetic
  equation.
\newblock {\em Mathematics In Engineering}, 6(4):536--558, 2024.

\bibitem{escobedo2024instability}
M.~Escobedo and A.~Menegaki.
\newblock Instability of singular equilibria of a wave kinetic equation.
\newblock {\em arXiv preprint arXiv:2406.05280}, 2024.

\bibitem{EscobedoBinh}
M.~Escobedo and M.-B. Tran.
\newblock Convergence to equilibrium of a linearized quantum {B}oltzmann
  equation for bosons at very low temperature.
\newblock {\em Kinetic and Related Models}, 8(3):493--531, 2015.

\bibitem{EscobedoVelazquez:2015:FTB}
M.~Escobedo and J.~J.~L. Vel{\'a}zquez.
\newblock Finite time blow-up and condensation for the bosonic {N}ordheim
  equation.
\newblock {\em Invent. Math.}, 200(3):761--847, 2015.

\bibitem{EscobedoVelazquez:2015:OTT}
M.~Escobedo and J.~J.~L. Vel{\'a}zquez.
\newblock On the theory of weak turbulence for the nonlinear {S}chr\"odinger
  equation.
\newblock {\em Mem. Amer. Math. Soc.}, 238(1124):v+107, 2015.

\bibitem{escobedo2024entropy}
M.l Escobedo, P.~Germain, J.~La, and A.~Menegaki.
\newblock Entropy maximizers for kinetic wave equations set on tori.
\newblock {\em arXiv preprint arXiv:2412.16026}, 2024.

\bibitem{fu2006optimality}
X.~Fu, X.~Zhang, and E.~Zuazua.
\newblock On the optimality of some observability inequalities for plate
  systems with potentials.
\newblock In {\em Phase space analysis of partial differential equations},
  pages 117--132. Springer, 2006.

\bibitem{GambaSmithBinh}
I.~M. Gamba, L.~M. Smith, and M.-B. Tran.
\newblock On the wave turbulence theory for stratified flows in the ocean.
\newblock {\em M3AS: Mathematical Models and Methods in Applied Sciences. Vol.
  30, No. 1 105-137}, 2020.

\bibitem{germain2017optimal}
P.~Germain, A.~D. Ionescu, and M.-B. Tran.
\newblock Optimal local well-posedness theory for the kinetic wave equation.
\newblock {\em Journal of Functional Analysis}, 279(4):108570, 2020.

\bibitem{germain2024stability}
P.~Germain, J.~La, and A.~Menegaki.
\newblock Stability of rayleigh-jeans equilibria in the kinetic fpu equation.
\newblock {\em arXiv preprint arXiv:2409.01507}, 2024.

\bibitem{hanche2010kolmogorov}
H.~Hanche-Olsen and H.~Holden.
\newblock The kolmogorov--riesz compactness theorem.
\newblock {\em Expositiones Mathematicae}, 28(4):385--394, 2010.

\bibitem{haraux1989series}
A.~Haraux.
\newblock S{\'e}ries lacunaires et contr{\^o}le semi-interne des vibrations
  d'une plaque rectangulaire.
\newblock {\em Journal de Math{\'e}matiques pures et appliqu{\'e}es},
  68(4):457--465, 1989.

\bibitem{hasselmann1962non}
K.~Hasselmann.
\newblock On the non-linear energy transfer in a gravity-wave spectrum part 1.
  general theory.
\newblock {\em Journal of Fluid Mechanics}, 12(04):481--500, 1962.

\bibitem{hasselmann1974spectral}
K.~Hasselmann.
\newblock On the spectral dissipation of ocean waves due to white capping.
\newblock {\em Boundary-Layer Meteorology}, 6(1-2):107--127, 1974.

\bibitem{hebey2008introduction}
E.~Hebey and B.~Pausader.
\newblock An introduction to fourth order nonlinear wave equations.
\newblock {\em Rn}, 2:H2, 2008.

\bibitem{lebeau1992control}
G.~Lebeau.
\newblock Control for hyperbolic equations.
\newblock {\em Journ{\'e}es {\'e}quations aux d{\'e}riv{\'e}es partielles},
  pages 1--24, 1992.

\bibitem{lee2005thermal}
G.~R. Lee-Dadswell, B.~G. Nickel, and C.~G. Gray.
\newblock Thermal conductivity and bulk viscosity in quartic oscillator chains.
\newblock {\em Physical Review E-Statistical, Nonlinear, and Soft Matter
  Physics}, 72(3):031202, 2005.

\bibitem{lefevere2004perturbative}
R.~Lefevere and A.~Schenkel.
\newblock Perturbative analysis of anharmonic chains of oscillators out of
  equilibrium.
\newblock {\em Journal of statistical physics}, 115:1389--1421, 2004.

\bibitem{lepri2000memory}
S.~Lepri.
\newblock Memory effects and heat transport in one-dimensional insulators.
\newblock {\em The European Physical Journal B-Condensed Matter and Complex
  Systems}, 18:441--446, 2000.

\bibitem{levandosky2000time}
S.~P. Levandosky and W.~A. Strauss.
\newblock Time decay for the nonlinear beam equation.
\newblock {\em Methods and Applications of Analysis}, 7(3):479--488, 2000.

\bibitem{lions1988controlabilite}
J.-L. Lions.
\newblock Contr{\^o}labilit{\'e} exacte, perturbations et stabilisation de
  syst{\`e}mes distribu{\'e}s. tome 1.
\newblock {\em RMA}, 8, 1988.

\bibitem{love2013treatise}
A.~E.~H. Love.
\newblock {\em A treatise on the mathematical theory of elasticity}.
\newblock Cambridge university press, 2013.

\bibitem{lukkarinen2008anomalous}
J.~Lukkarinen and H.~Spohn.
\newblock Anomalous energy transport in the fpu-$\beta$ chain.
\newblock {\em Communications on Pure and Applied Mathematics: A Journal Issued
  by the Courant Institute of Mathematical Sciences}, 61(12):1753--1786, 2008.

\bibitem{LukkarinenSpohn:WNS:2011}
J.~Lukkarinen and H.~Spohn.
\newblock Weakly nonlinear {S}chr\"odinger equation with random initial data.
\newblock {\em Invent. Math.}, 183(1):79--188, 2011.

\bibitem{menegaki20222}
A.~Menegaki.
\newblock L2-stability near equilibrium for the 4 waves kinetic equation.
\newblock {\em Kinetic and Related Models}, 17(4):514--532, 2024.

\bibitem{Nazarenko:2011:WT}
S.~Nazarenko.
\newblock {\em Wave turbulence}, volume 825 of {\em Lecture Notes in Physics}.
\newblock Springer, Heidelberg, 2011.

\bibitem{newell2011wave}
A.~C. Newell and B.~Rumpf.
\newblock Wave turbulence.
\newblock {\em Annual review of fluid mechanics}, 43:59--78, 2011.

\bibitem{newell2013wave}
A.~C. Newell and B.~Rumpf.
\newblock Wave turbulence: a story far from over.
\newblock In {\em Advances in wave turbulence}, pages 1--51. World Scientific,
  2013.

\bibitem{nguyen2017quantum}
T.~T. Nguyen and M.-B. Tran.
\newblock On the {K}inetic {E}quation in {Z}akharov's {W}ave {T}urbulence
  {T}heory for {C}apillary {W}aves.
\newblock {\em SIAM J. Math. Anal.}, 50(2):2020--2047, 2018.

\bibitem{ToanBinh}
T.~T. Nguyen and M.-B. Tran.
\newblock Uniform in time lower bound for solutions to a quantum boltzmann
  equation of bosons.
\newblock {\em Archive for Rational Mechanics and Analysis}, 231(1):63--89,
  2019.

\bibitem{pausader2007scattering}
B.~Pausader.
\newblock Scattering and the levandosky--strauss conjecture for fourth-order
  nonlinear wave equations.
\newblock {\em Journal of Differential Equations}, 241(2):237--278, 2007.

\bibitem{pausader2009analyticity}
B.~Pausader and W.~Strauss.
\newblock Analyticity of the scattering operator for the beam equation.
\newblock {\em Discrete Contin. Dyn. Syst}, 25:617--626, 2009.

\bibitem{Peierls:1993:BRK}
R.~Peierls.
\newblock Zur kinetischen theorie der warmeleitung in kristallen.
\newblock {\em Annalen der Physik}, 395(8):1055--1101, 1929.

\bibitem{Peierls:1960:QTS}
R.~E. Peierls.
\newblock Quantum theory of solids.
\newblock In {\em Theoretical physics in the twentieth century ({P}auli
  memorial volume)}, pages 140--160. Interscience, New York, 1960.

\bibitem{PomeauBinh}
Y.~Pomeau and M.-B. Tran.
\newblock Statistical physics of non equilibrium quantum phenomena.
\newblock {\em Lecture Notes in Physics, Springer}, 2019.

\bibitem{smith2002generation}
L.~M. Smith and F.~Waleffe.
\newblock Generation of slow large scales in forced rotating stratified
  turbulence.
\newblock {\em Journal of Fluid Mechanics}, 451:145--168, 2002.

\bibitem{Binh1}
A.~Soffer and M.-B. Tran.
\newblock On the dynamics of finite temperature trapped bose gases.
\newblock {\em Advances in Mathematics}, 325:533--607, 2018.

\bibitem{soffer2019energy}
A.~Soffer and M.-B. Tran.
\newblock On the energy cascade of 3-wave kinetic equations: beyond
  kolmogorov--zakharov solutions.
\newblock {\em Communications in Mathematical Physics}, pages 1--48, 2019.

\bibitem{Spohn:TPB:2006}
H.~Spohn.
\newblock The phonon {B}oltzmann equation, properties and link to weakly
  anharmonic lattice dynamics.
\newblock {\em J. Stat. Phys.}, 124(2-4):1041--1104, 2006.

\bibitem{staffilani2024condensation}
G.~Staffilani and M.-B. Tran.
\newblock Condensation and non-condensation times for 4-wave kinetic equations.
\newblock {\em arXiv preprint arXiv:2407.18533}, 2024.

\bibitem{staffilani2024energy}
G.~Staffilani and M.-B. Tran.
\newblock On the energy transfer towards large values of wavenumbers for
  solutions of 4-wave kinetic equations.
\newblock {\em arXiv preprint arXiv:2407.18508}, 2024.

\bibitem{CraciunSmithBoldyrevBinh}
M.-B. Tran, G.~Craciun, L.~M. Smith, and S.~Boldyrev.
\newblock A reaction network approach to the theory of acoustic wave
  turbulence.
\newblock {\em Journal of Differential Equations}, 269(5):4332--4352, 2020.

\bibitem{zakharov2012kolmogorov}
V.~E. Zakharov, V.~S. L'vov, and G.~Falkovich.
\newblock {\em Kolmogorov spectra of turbulence I: Wave turbulence}.
\newblock Springer Science \& Business Media, 2012.

\bibitem{zuazua1987controlabilite}
E.~Zuazua and J.-L. Lions.
\newblock Contr{\^o}labilit{\'e} exacte d'un mod{\`e}le de plaques vibrantes en
  un temps arbitrairement petit.
\newblock {\em Comptes rendus de l'Acad{\'e}mie des sciences. S{\'e}rie 1,
  Math{\'e}matique}, 304(7):173--176, 1987.

\end{thebibliography}

\end{document}